\documentclass[11pt,twoside]{amsart}
\usepackage{amsmath, amsthm, amscd, amsfonts, amssymb, graphicx, color}
\usepackage[bookmarksnumbered, colorlinks, plainpages]{hyperref}

\newtheorem{thm}{Theorem}[section]
\newtheorem{cor}[thm]{Corollary}

\newtheorem{defn}[thm]{Definition}

\newtheorem{exm}[thm]{\bf{Example}}

\numberwithin{equation}{section}
\def\pn{\par\noindent}

\begin{document}


\vspace{1.3 cm}

\title{Approximation of a Cauchy-Jensen Additive Mapping in Various Normed Spaces}
\author{H. Azadi Kenary and Th.M. Rassias}

\thanks{{\scriptsize
\hskip -0.4 true cm MSC(2010): 39B22, 39B52, 39B82, 46S10, 47S10, 46S40.
$*$Corresponding author
\newline\indent{\scriptsize }}}

\maketitle


\begin{abstract}  In this paper, using the fixed point and direct methods, we prove the generalized Hyers-Ulam-Rassias stability
of the following Cauchy-Jensen additive functional equation:
\begin{equation}\label{main}
f\left(\frac{x+y+z}{2}\right)+f\left(\frac{x-y+z}{2}\right)=f(x)+f(z)\end{equation}
in various normed spaces.\\
The concept of Hyers-Ulam-Rassias stability originated from Th. M.
Rassias stability theorem that appeared in his paper: On the
stability of the linear mapping in Banach spaces, Proc. Amer.
Math. Soc. 72 (1978), 297-300.
\end{abstract}

\vskip 0.2 true cm


\pagestyle{myheadings}
\markboth{\rightline {\scriptsize  H. Azadi Kenary and Th.M. Rassias}}
         {\leftline{\scriptsize Stability of functional equations}}

\bigskip
\bigskip


\section{\bf Introduction}
\vskip 0.4 true cm

A classical question in the theory of functional equations is the
following: \textit{When is it true that a function which approximately
satisfies a functional equation must be close to an exact solution
of the equation?}. If the problem accepts a solution, we say that
the equation is {\it stable}. The first stability problem
concerning group homomorphisms was raised by Ulam \cite{u} in
1940. In the next year, Hyers \cite{hy} gave a positive answer to
the above question for additive groups under the assumption that
the groups are Banach spaces. In 1978, Rassias \cite{ra} proved a
generalization of Hyers's theorem for linear mappings.
\begin{thm}(Th.M. Rassias): Let $f: E\rightarrow E'$ be a mapping
from a normed vector space $E$ into a Banach space $E'$ subject to
the inequality
\begin{equation}\nonumber
\|f(x+y)-f(x)-f(y)\|\leq \epsilon(\|x\|^p+ \|y\|^p)
\end{equation}
for all $x,y\in E$, where $\epsilon$ and $p$ are constants with
$\epsilon>0$ and $0\leq p<1$. Then the limit
$$L(x)=\lim_{n\rightarrow \infty}\frac{f(2^n x)}{2^n}$$ 
exists for
all $x\in E$ and $L: E\rightarrow E'$ is the unique additive
mapping which satisfies
\begin{equation}\nonumber
\|f(x)-L(x)\|\leq \frac{2\epsilon}{2-2^p}\|x\|^p\:,\ \ \text{for all $x\in E$. }
\end{equation}
Also, if for each $x\in E$ the function $f(tx)$
is continuous in $t\in \mathbb{R}$, then $L$ is linear.
\end{thm}
This new
concept is known as generalized Hyers-Ulam stability or
Hyers-Ulam-Rassias stability of functional equations. Furthermore,
in 1994, a generalization of Rassias's theorem was obtained by
G\v{a}vruta \cite{g} by replacing the bound
$\epsilon(\|x\|^p+\|y\|^p)$ by a general control function
$\varphi(x,y)$.\\
In 1983, a generalized Hyers-Ulam stability problem for the
quadratic functional equation was proved by Skof \cite{s} for
mappings $f:X\rightarrow Y$, where $X$ is a normed space and $Y$
is a Banach space. In 1984, Cholewa \cite{ch} noticed that the
theorem of Skof is still true if the relevant domain $X$ is
replaced by an Abelian group and, in 2002, Czerwik \cite{c} proved
the generalized Hyers-Ulam stability of the quadratic functional
equation. The reader is referred to (\cite{a}-\cite{reza})
 and references therein for detailed information on stability of functional equations
 .\\
In 1897, Hensel \cite{h} has introduced a normed space which does
not have the Archimedean property.
 It turned out that
non-Archimedean spaces have many nice applications (see \cite{d,
kb, kh, ko, n}).\\
Katsaras \cite{Kat} defined a fuzzy norm on a vector space to
construct a fuzzy vector topological structure on the space. Some
mathematicians have defined fuzzy norms on a vector space from
various points of view (cf. \cite{f}, \cite{k-s},
\cite{c-3}).\\
In particular, Bag and Samanta \cite{B-S}, following Cheng and
Mordeson \cite{C-M}, introduced an idea of fuzzy norm in such a manner
that the corresponding fuzzy metric is of Karmosil and Michalek
type \cite{kar-J.M}. They established a decomposition theorem of a
fuzzy norm into a family of crisp norms and investigated some
properties of fuzzy normed spaces \cite{B-S2}.
\begin{defn}By a \emph{non-Archimedean
field} we mean a field $\mathbb{K}$ equipped with a function
(valuation) $|\cdot|: \mathbb{K} \rightarrow [0,\infty)$ such
that, for all $r,s\in \mathbb{K}$, the following conditions hold:

$(a)$ $|r|=0$ if and only if $r=0$;

$(b)$ $|rs|=|r||s|$;

$(c)$ $|r+s|\leq \max\{|r|,|s|\}.$
\end{defn}

Clearly, by (b), $|1| = |-1|=1$ and so, by induction, it follows
from (c) that $|n|\leq 1$ for all $n\geq 1$.

\begin{defn} Let $X$ be a vector space
over a scalar field $\mathbb{K}$ with a non-Archimedean
non-trivial valuation $|\cdot|$.

(1) A function $\|\cdot\|: X\rightarrow \mathbb{R}$ is a
\emph{non-Archimedean norm} (valuation) if it satisfies the
following conditions:

$(a)$ $\|x\|=0$ if and only if $x=0$ for all $x\in X$;

$(b)$ $\|rx\|=|r|\|x\|$ for all $r\in \mathbb{K}$ and $x\in X$;

$(c)$ the strong triangle inequality (ultra-metric) holds, that
is,
$$\|x+y\|\leq \max\{\|x\|,\|y\|\}$$
for all $x,y\in X$.

(2) The space $(X,\|\cdot\|)$ is called a {\it non-Archimedean
normed space}.
\end{defn}
Note that
$$||x_n -x_m||\leq max\{||x_{j+1}-x_j||: m\leq j\leq n-1\}\:,$$
for all $m,n\in \mathbb{N}$ with $n>m$.
\begin{defn} Let $(X,\|\cdot\|)$ be a
non-Archimedean normed space.

$(a)$ A sequence $\{x_n\}$ is a \emph{Cauchy sequence} in $X$ if
$\{x_{n+1}-x_n\}$ converges to zero in $X$.

$(b)$ The non-Archimedean normed space $(X,\|\cdot\|)$ is said to
be {\it complete} if every Cauchy sequence in $X$ is convergent.\end{defn}

The most important examples of non-Archimedean spaces are $p$-adic
numbers. A key property of $p$-adic numbers is that they do not
satisfy the Archimedean axiom: for all $x,y>0$, there exists a
positive integer $n$ such that $x<ny$.\vskip 2mm

\begin{exm}Fix a prime number $p$. For any
nonzero rational number $x$, there exists a unique positive
integer $n_x$ such that $x=\frac{a}{b}p^{n_x}$, where $a$ and $b$
are positive integers not divisible by $p$. Then $|x|_p:=p^{-n_x}$
defines a non-Archimedean norm on $\mathbb{Q}$. The completion of
$\mathbb{Q}$ with respect to the metric $d(x,y)=|x-y|_p$ is
denoted by $\mathbb{Q}_p$, which is called the \emph{$p$-adic
number field}. In fact, $\mathbb{Q}_p$ is the set of all formal
series $$x=\sum_{k\geq n_x}^{\infty}a_{k} p^{k}\:,$$ where $|a_k|\leq
p-1$. The addition and multiplication between any two elements of
$\mathbb{Q}_p$ are defined naturally. The norm $$\left|\sum_{k\geq
n_x}^{\infty}a_{k} p^{k}\right|_p=p^{-n_x}$$ is a non-Archimedean norm on
$\mathbb{Q}_p$ and  $\mathbb{Q}_p$ is a locally compact filed.\end{exm}

In section 3, we  adopt the usual terminology, notions and
conventions of the theory of random normed spaces as in
\cite{s-s}.\\
 Throughout this paper, let $\bigtriangleup^{+}$ denote the set of all
probability distribution functions $F:\mathbb{R} \cup
[-\infty,+\infty]\rightarrow [0,1]$ such that $F$ is
left-continuous and nondecreasing on $\mathbb{R}$ and $F(0)=0,
F(+\infty)=1$. It is clear that the set
$$D^+:=\{F\in \bigtriangleup^+: l^{-}F(-\infty)=1\},$$ where
$l^{-}f(x)=\lim_{t\rightarrow x^{-}}f(t)$, is a subset of
$\bigtriangleup^+$. The set $\bigtriangleup^+$ is partially
ordered by the usual point-wise ordering of functions, that is,
$F\leq G $ if and only if $F(t)\leq G(t)$ for all $t\in
\mathbb{R}$. For any $a\geq 0$,  the element $H_a(t)$ of $D^+$ is
defined by
\begin{equation}\nonumber
H_{a}(t) := \left \{\begin{array}{lll}
  0, &  \mbox{if} & t\leq a,\\
  1, & \mbox{if} & t>a.\\
  \end{array}\right.
\end{equation}

We can easily show that the maximal element in $\bigtriangleup^+$
 is the distribution function $H_0(t)$.\vskip 2mm

\noindent {\bf Definition 2.1.}  {\rm A function $T:[0,1]^2
\rightarrow [0,1]$ is a {\it continuous triangular norm} (briefly,
a $t$-norm) if $T$ satisfies the following conditions:

$(a)$ $T$ is commutative and associative;

$(b)$ $T$ is continuous;

$(c)$ $T(x,1)=x$ for all $x\in[0,1]$;

$(d)$ $T(x,y)\leq T(z,w)$ whenever $x \leq z$ and $y \leq w$ for
all elements \mbox{$x,y,z,w \in [0,1]$.}} \vskip 2mm

Three typical examples of continuous $t$-norms are the
following:
$$T(x,y)=xy,~ T(x,y)=\max\{a+b-1,0\},~
T(x,y)=\min(a,b).$$ 
Recall
that, if $T$ is a $t$-norm and $\{x_n\}$ is a sequence in $[0,1]$,
then $T^{n}_{i=1}x_i$ is defined recursively by 
$$T^{1}_{i=1}x_1=
x_1\ \ \text{and}\ \  T^{n}_{i=1}x_i=T(T^{n-1}_{i=1}x_i,x_n)\ \ \text{for all $n\geq
2$.}$$ 
$T_{i=n}^\infty x_i$ is defined by $T_{i=1}^\infty
x_{n+i}$.\vskip 2mm

\begin{defn} A {\it random normed space}
(briefly, $RN$-space) is a triple $(X,\mu,T)$, where $X$ is a
vector space, $T$ is a continuous $t$-norm and $\mu: X
\rightarrow D^+$ is a mapping such that the following conditions
hold:\\
$(a)$ $\mu_x(t)=H_{0}(t)$ for all $t>0$ if and only if $x=0$;\\
$(b)$ $\mu_{\alpha x}(t)=\mu_{x}\left(\frac{t}{|\alpha|}\right)$ for all
$\alpha \in \mathbb{R}$ with $\alpha \neq 0$, $x\in X$ and $t\geq
0$;\\
$(c)$ $\mu_{x+y}(t+s)\geq T(\mu_x(t),\mu_y(s))$ for all $x,y\in
X$ and $t,s\geq0$. \end{defn}

Every normed space $(X,\|\cdot\|)$ defines a random normed space
$(X,\mu,T_M)$, where
$$\mu_u(t)=\frac{t}{t+\|u\|}\ \  \text{for all}\ t>0$$ and $T_M$ is the minimum $t$-norm.
This space $X$ is called the {\it induced random normed space}.\\
If the $t$-norm $T$ is such that $\sup_{0<a<1}T(a,a)=1$, then
every $RN$-space $(X,\mu,T )$ is a metrizable linear topological
space with the topology $\tau$ (called the {\it $\mu$-topology}
or the {\it $(\epsilon,\delta)$-topology}, where $\epsilon>0$ and
$\lambda\in (0,1)$) is induced by the base $\{U(\epsilon,\lambda)\}$
of neighborhoods of $\theta$, where
$$U(\epsilon,\lambda)=\{x\in X : \mu_x(\epsilon)
>1-\lambda\}.$$

\begin{defn} Let $(X,\mu,T)$ be an
RN-space.\\
 $(a)$ A sequence $\{x_n\}$ in $X$ is said to be
{\it convergent} to a point $x\in X$ (write $x_n \rightarrow x$ as
$n\to\infty$) if $$\lim_{n \rightarrow \infty}\mu_{x_n -x}(t)=1$$
for all $t>0$.\\
$(b)$ A sequence $\{x_n\}$ in $X$ is called a {\it Cauchy
sequence} in $X$ if  $$\lim_{n \rightarrow \infty}\mu_{x_n
-x_m}(t)=1$$ for all $t>0$.\\
$(c)$ The $RN$-space $(X,\mu,T)$ is said to be {\it complete} if
every Cauchy sequence in $X$ is convergent.\end{defn}

\begin{thm}If $(X,\mu,T)$ is RN-space and $\{x_n\}$ is
a sequence such that $x_n \rightarrow x$, then $\lim_{n\rightarrow
\infty}\mu_{x_n}(t)=\mu_{x}(t)$.
\end{thm}

\begin{defn} Let $X$ be a real vector space. A function \mbox{$N: X\times
\mathbb{R}\rightarrow [0,1]$} is called a fuzzy norm on $X$ if for
all $x,y\in X$ and all $s,t\in \mathbb{R}$, \\
$(N1)$~~ $N(x,t)=0$ for $t\leq 0$;\\
$(N2)$~~ $x=0$ if and only if $N(x,t)=1$ for all $t>0$;\\
$(N3)$~~$N(cx,t)=N\left(x,\frac{t}{|c|}\right)$ if $c\neq 0$;\\
$(N4)$~~$N(x+y,c+t)\geq min\{N(x,s),N(y,t)\}$;\\
$(N5)$~~$N(x,.)$ is a non-decreasing function of $\mathbb{R}$ and
$\lim_{t\rightarrow \infty}N(x,t)=1$;\\
$(N6)$~~ for $x\neq 0$, $N(x,.)$ is continuous on $\mathbb{R}$.
\end{defn}
The pair $(X,N)$ is called a fuzzy normed vector space.

\begin{exm}
Let $(X,\|.\|)$ be a normed linear space and $\alpha,\beta>0$.
Then
\begin{equation}\nonumber
N(x,t) = \left \{\begin{array}{lll}
  \frac{\alpha t}{\alpha t +\beta \|x\|} &\text{if}\ \  t>0, x\in X\\
  0  &\text{if}\ \ t\leq 0, x\in X\\
  \end{array}\right.
\end{equation}
is a fuzzy norm on $X$.
\end{exm}

\begin{defn} Let $(X,N)$ be a fuzzy normed vector space. A sequence
$\{x_n\}$ in $X$ is said to be convergent or converge if there
exists an $x\in X$ such that $\lim_{t\rightarrow \infty}N(x_n
-x,t)=1$ for all $t>0$. In this case, $x$ is called the limit of
the sequence $\{x_n\}$ in $X$ and we denote it by
$$N-\lim_{t\rightarrow \infty}x_n =x\:.$$
\end{defn}
\begin{defn} Let $(X,N)$ be a fuzzy normed vector space. A sequence
$\{x_n\}$ in $X$ is called Cauchy if for each $\epsilon >0$ and
each $t>0$ there exists an $n_0 \in \mathbb{N}$ such that for all
$n\geq n_0$ and all $p>0$, we have $$N(x_{n+p}-x_n,t)>1-\epsilon\:.$$
\end{defn}
It is well known that every convergent sequence in a fuzzy normed
vector space is Cauchy. If each Cauchy sequence is convergent,
then the fuzzy norm is said to be complete and the fuzzy normed
vector space is called a fuzzy Banach space.\\
We say that a mapping $f:X\rightarrow Y$ between fuzzy normed
vector spaces $X$ and $Y$ is continuous at a point $x\in X$ if for
each sequence $\{x_n\}$ converging to $x_0 \in X$, then the
sequence $\{f(x_n)\}$ converges to $f(x_0)$. If $f:X\rightarrow Y$
is continuous at each $x\in X$, then $f:X\rightarrow Y$ is said to
be continuous on $X$.

\begin{defn}Let $X$ be a set. A function
$d:X\times X\rightarrow[0,\infty]$ is called a generalized metric
on $X$ if $d$ satisfies the following conditions:\\
$(a)$ $d(x,y)=0$ if and only if $x=y$ for all $x,y\in X$;\\
$(b)$ $d(x,y)=d(y,x)$ for all $x,y\in X$;\\
$(c)$ $d(x,z)\leq d(x,y)+d(y,z)$ for all $x,y,z\in X$.\end{defn}

\begin{thm}\label{thras}Let (X,d) be a complete generalized metric space
and $J: X\rightarrow X$ be a strictly contractive mapping with
Lipschitz constant \mbox{$L<1$.} Then, for all $x\in X$, either
$
d(J^n x,J^{n+1}x)=\infty
$
for all nonnegative integers $n$ or there exists a positive
integer $n_0$ such that\\
$(a)$ $d(J^n x,J^{n+1}x)<\infty$ for all $n_0\geq n_0$;\\
$(b)$ the sequence $\{J^n x\}$ converges to a fixed point $y^*$ of
$J$;\\
$(c)$ $y^*$ is the unique fixed point of $J$ in the set $Y=\{y\in
X : d(J^{n_0}x,y)<\infty\}$;\\
$(d)$ 
$$d(y,y^*)\leq \frac{d(y,Jy)}{1-L}\ \ \text{for all $y\in Y$.}$$
\end{thm}


\section {\bf Non-Archimedean Stability of the Functional Equation (\ref{main})}

In this section, we deal with the stability problem for the
Cauchy-Jensen additive functional equation (\ref{main}) in non-Archimedean normed spaces.
\subsection{ A Fixed Point Approach}
\begin{thm}
Let $X$ is a non-Archimedean normed space and that $Y$ be a
complete non-Archimedean space. Let $\varphi : X^{3}
\rightarrow [0, \infty)$  be a function such that there exists an
$\alpha <1$ with
\begin{eqnarray}\label{tt}
&&\varphi\left(\frac{x}{2}, \frac{y}{2},
\frac{z}{2}\right) \le
\frac{\alpha\varphi\left(x,y,z\right)}{ |2|}
\end{eqnarray}
for all $x,y,z \in X$. Let $f : X \rightarrow Y$ be a
mapping satisfying
\begin{eqnarray}\label{bb}
&&\left\|f\left(\frac{x+y+z}{2}\right)+f\left(\frac{x-y+z}{2}\right)-f(x)-f(z)\right\|_Y\le  \varphi(x,y,z)
\end{eqnarray}
for all $x,y,z\in X$. Then there exists a unique
 additive mapping \mbox{$L : X \rightarrow Y$}  such that
\begin{eqnarray}\label{yy}
\|f(x)-L(x)\|_Y \le \frac{\alpha \varphi(x,2x,x)}{|2| - |2|\alpha}
\end{eqnarray}
for all $x\in X$.
\end{thm}

\begin{proof}
Setting $y=2x$ and $z=x$ in {\rm (2.2)},
 we get
\begin{eqnarray}\label{kk}
\left\|f(2x)-2f(x)
\right\|_Y \le \varphi(x, 2x,x)\:,\ \ \text{for all $x \in X$. }
\end{eqnarray}
Thus
\begin{eqnarray}\label{cc}
\left\|f(x) - 2f\left(\frac{x}{2}\right)
\right\|_Y \le \varphi\left(\frac{x}{2},
x,\frac{x}{2}\right) \le\frac{\alpha \varphi(x,2x,x)}{ |2|}
\end{eqnarray}
for all $x \in X$. Consider the set $ S :  = \{ h : X \rightarrow Y\}$ and
introduce the  generalized metric  on $S$:
\begin{eqnarray*}
&& d(g, h) := \inf \Big\{ \mu \in (0,+\infty) :\| g(x)-h(x)\|_Y \le
\mu \varphi(x, 2x,x), ~~
\forall x \in X \Big\},
\end{eqnarray*} where, as usual, $\inf \phi = +\infty$.
It is easy to show that $(S, d)$ is complete (see  \cite{mr}). Now
we consider the linear mapping $J: S \rightarrow S$ such that
$$J g(x): =  2g\left(\frac{x}{2}\right)\:, \ \ \text{for all $x \in X$.}
$$
 Let $g, h \in S$ be given such that $d(g, h) =  \varepsilon$. Then
$$\|g(x)-h(x)\|_Y \le  \epsilon \varphi(x, 2x,x)\:,\ \ \text{for all $x \in X$. }$$
Hence
\begin{eqnarray*} \| J g(x) - J h(x)\|_Y
&=&   \left\| 2g\left(\frac{x}{2}\right)
-
 2h\left(\frac{x}{2}\right)
\right\|_Y = |2|\left\|
 g\left(\frac{x}{2}\right)
 -
 h\left(\frac{x}{2}\right)
\right\|_Y \\ \nonumber &\leq&
|2|\varphi\left(\frac{x}{2},x,
\frac{x}{2}\right) \le \alpha
\cdot\epsilon \varphi(x, 2x,x)
\end{eqnarray*}
for all $x\in X$.\\ 
Thus $d(g, h) = \varepsilon$ implies that $d(J g,
J h)\le \alpha\varepsilon$. This means that
$$d(J g, Jh) \le \alpha d(g, h)\:,\ \ \text{for all $g, h\in S$.}$$ 
It follows from (\ref{cc})  that  
$$d(f, Jf) \le
\frac{\alpha}{|2|}\:.$$
By Theorem \ref{thras}, there exists a mapping $L : X \rightarrow
Y$ satisfying the following:

(1) $L$ is a fixed point of $J$, i.e.,
\begin{eqnarray}\label{ff}
\frac{L(x)}{2}  = L\left(\frac{x}{2}\right)\:,\ \ \text{for all $x \in X$.}
\end{eqnarray}
The mapping $L$ is a unique fixed point of $J$ in the set
$$M=\{g\in S : d(h, g) < \infty\} .$$ This implies that $L$ is a
unique mapping satisfying (\ref{ff}) such that there exists a
$\mu \in (0, \infty)$ satisfying
$$ \|f(x)- L(x)\|_Y\le \mu   \varphi(x, 2x,x)\:,\ \ \text{for all $x \in
X$;}$$\\
(2) $d(J^n f, L) \rightarrow 0$ as $n \rightarrow \infty$. This
implies the equality
\begin{eqnarray}\label{limit}
\lim_{n\to \infty} 2^n
f\left(\frac{x}{2^n} \right) = L(x)
\end{eqnarray}
for all $x \in X$;\\
(3) $$d(f, L) \le \frac{1}{1-\alpha} d(f, Jf)\:,$$ which implies the
inequality
$$d(f, L) \le \frac{\alpha}{|2|- |2|\alpha}.$$ 
This implies that the inequalities (\ref{yy}) hold.

It follows from (\ref{tt}) and (\ref{bb}) that
\begin{eqnarray*}
&& \left\|L\left(\frac{x+y+z}{2}\right)+L\left(\frac{x-y+z}{2}\right)-L(x)-L(z)\right\|_Y
\\ &&  = \lim_{n\to\infty} |2|^{n}\left\|f\left(\frac{x+y+z}{2^{n+1}}\right)+f\left(\frac{x-y+z}{2^{n+1}}\right)-f\left(\frac{x}{2^{n}}\right)
-f\left(\frac{z}{2^{n}}\right)\right\|_Y
\\ && \le  \lim_{n\to\infty} |2|^n \varphi\left(\frac{x}{2^n},\frac{y}{2^n},
\frac{z}{2^n}\right)
\le  \lim_{n\to\infty} |2|^n. \frac{\alpha^n  \varphi(x,y,z)}{|2|^n}
=0
\end{eqnarray*}
 for all $x,y,z\in X$ . Thus
 $$ L\left(\frac{x+y+z}{2}\right)+L\left(\frac{x-y+z}{2}\right)=L(x)+L(z)$$
for all $x,y,z\in X$. Hence $L: X \rightarrow Y$ is a
Cauchy-Jensen  mapping.   It follows from (\ref{cc}) and (\ref{limit}) that
\begin{eqnarray*}
\left\|2 L\left(\frac{x}{2}\right) - L( x) \right\|_Y &=&
\lim_{n\to\infty}|2|^n \left\|2
f\left(\frac{x}{2^{n+1}}\right)-f\left(\frac{x}{2^n}\right)
\right\|_Y\\ 
&\le&  \lim_{n\to\infty} |2|^n \varphi\left(\frac{x}{2^{n+1}},
\frac{x}{2^{n}},\frac{x}{2^{n+1}}\right)
\\ \nonumber   &\le&  \lim_{n\to\infty} |2|^n
. \frac{ \alpha^n
\varphi(x,2x,x)}{|2|^n}=0
\end{eqnarray*}
 for all $x \in X$. Therefore $$2  L\left(\frac{x}{2}\right) - L(x) =0$$
for all $x \in X$. Hence $L: X \rightarrow Y$ is additive and we
get the desired results.
\end{proof}

\begin{cor}
Let $\theta$ be a positive real number and $r$ is a real number
with $0< r<1$. Let $f : X \rightarrow Y$ be a mapping  satisfying
\begin{eqnarray}\nonumber
\left\|f\left(\frac{x+y+z}{2}\right)+f\left(\frac{x-y+z}{2}\right)-f(x)-f(z)\right\|_Y\leq \theta \left(\|x\|^r +\|y\|^r+ \|z\|^r\right)
\end{eqnarray}
for all $x,y,z\in X$ . Then there exists a unique
 additive mapping \mbox{$L : X \rightarrow Y$}  such that
\begin{eqnarray*}
\|f(x)-L(x)\|_Y \le \frac{|2|\theta(2+|2|^r )\|x\|^r}{|2|^{r+1}-|2|^2}\:,\ \ \text{for all $x\in X$.}
\end{eqnarray*}
\end{cor}

\begin{proof}
The proof follows from Theorem 2.1 by setting
$$\varphi(x,y,z)=\left(\|x\|^r +\|y\|^r+ \|z\|^r\right)$$
for all $x,y,z \in X$. Then we can choose $$\alpha=|2|^{1-r}$$
  and we get the desired result.
\end{proof}

\begin{thm} Let $X$ be a non-Archimedean normed space and that $Y$ be a
complete non-Archimedean space. Let $\varphi : X^{3}
\rightarrow [0, \infty)$  be a function such that there exists an
$\alpha <1$ with
\begin{eqnarray*}
 \varphi\left(x,y,z\right)\le |2|\alpha\varphi\left(\frac{x}{2},
\frac{y}{2},\frac{z}{2}\right )\:,\ \ \text{for all $x,y,z\in X$. }
\end{eqnarray*}
Let $f : X \rightarrow Y$ be a
mapping satisfying (\ref{bb}). Then there exists a
unique  additive mapping $L : X \rightarrow Y$  such
that
\begin{eqnarray}\label{yyx}
\|f(x)-L(x)\|_Y \le \frac{\varphi(x, 2x,x)}{|2|-|2|\alpha}\:,\ \ \text{for all $x\in X$.}
\end{eqnarray}
\end{thm}

\begin{proof}
Let $(S, d)$ be the generalized metric space defined in the proof
of Theorem 2.1. Now we consider the linear mapping $J: S \rightarrow S$ such that
$$J g(x): = \frac{g(2x)}{2}\:,\ \ \text{for all $x \in X$.}$$
Let $g, h \in S$ be given such that $d(g, h) =  \varepsilon$. Then
$$\|g(x)-h(x)\|_Y \le  \epsilon \varphi(x, 2x,x)\:,\ \ \text{for all $x \in X$. }$$
Hence
\begin{eqnarray*} \| J g(x) - J h(x)\|_Y
&=&   \left\| \frac{g(2x)}{2}
-
 \frac{h(2x)}{2}
\right\|_Y = \frac{\left\|g(2x)
 -
 h(2x)\right\|_Y}{|2|}
 \\ \nonumber &\leq&
\frac{\varphi\left(2x,4x,
2x\right)}{|2|} \le \frac{|2|\alpha
\cdot\epsilon \varphi(x, 2x,x)}{|2|}
\end{eqnarray*}
for all $x\in X$. Thus $d(g, h) = \varepsilon$ implies that $$d(J g,
J h)\le \alpha\varepsilon\:.$$ This means that
$$d(J g, Jh) \le \alpha d(g, h)\:,\ \ \text{for all $g, h\in S$.}$$ 
It follows from (\ref{kk})  that  
$$d(f, Jf) \le
\frac{1}{|2|}\:.$$
By Theorem \ref{thras}, there exists a mapping $L : X \rightarrow
Y$ satisfying the following:

(1) $L$ is a fixed point of $J$, i.e.,
\begin{eqnarray}\label{ffx}
L(2x)  = 2L\left(x\right)\:,\ \ \text{ for all $x \in X$.}
\end{eqnarray}
The mapping $L$ is a unique fixed point of $J$ in the set
$$M=\{g\in S : d(h, g) < \infty\} .$$ 
This implies that $L$ is a
unique mapping satisfying (\ref{ffx}) such that there exists a
$\mu \in (0, \infty)$ satisfying
$$ \|f(x)- L(x)\|_Y\le \mu   \varphi(x, 2x,x)\:,\ \ \text{for all $x \in
X$;}$$
(2) $d(J^n f, L) \rightarrow 0$ as $n \rightarrow \infty$. This
implies the equality
\begin{eqnarray}\label{limit}
\lim_{n\to \infty}\frac{f(2^n x)}{2^n} = L(x)
\end{eqnarray}
for all $x \in X$;\\
(3) $d(f, L) \le \frac{1}{1-\alpha} d(f, Jf)$, which implies the
inequality
$d(f, L) \le \frac{1}{|2|- |2|\alpha}.$ This implies that the inequalities (\ref{yyx}) holds.
 The rest of the proof is similar to the proof of Theorem 2.1.
\end{proof}

\begin{cor}
Let $\theta$ be a positive real number and $r$ be a real number
with $r> 1$. Let $f : X \rightarrow Y$ be a mapping satisfying
 \begin{eqnarray}\nonumber
\left\|f\left(\frac{x+y+z}{2}\right)+f\left(\frac{x-y+z}{2}\right)-f(x)-f(z)\right\|_Y\leq \theta \left(\|x\|^r +\|y\|^r+ \|z\|^r\right)
\end{eqnarray}
for all $x,y,z\in X$. Then there exists a unique
additive  mapping $L : X \rightarrow Y$  such that
\begin{eqnarray*}
\|f(x)-L(x)\|_Y \le \frac{\theta(2+|2|^r )\|x\|^r}{|2|-|2|^r
}\:,\ \ \text{for all $x\in X$.}
\end{eqnarray*}
\end{cor}

\begin{proof}
The proof follows from Theorem 2.3 by taking
$$\varphi(x,y,z)=\left(\|x\|^r +\|y\|^r+ \|z\|^r\right)$$
for all $x,y,z \in X$. Then we can choose $\alpha=|2|^{r-1}$
  and we get the desired result.
\end{proof}

\subsection {\bf  A Direct Method}
\begin{equation*}\end{equation*}

 In this section, using the direct method, we prove the generalized
Hyers-Ulam-Rassias stability of the Cauchy-Jensen additive functional
equation (\ref{main}) in non-Archimedean space .
\begin{thm}\label{th3.1}
Let $G$ be an additive semigroup and that $X$ is a non-Archimedean
Banach space. Assume that $\zeta: G^{3} \rightarrow
[0,+\infty)$ be a function such that
\begin{equation}\label{3.1}
\lim_{n\rightarrow
\infty}|2|^{n}\zeta\left(\frac{
x}{2^{n}},\frac{y}{2^{n}},\frac{z}{2^{n}}\right)=0
\end{equation}
for all $x,y,z\in G$. Suppose that, for any $x\in G$, the
limit
\begin{equation}\label{3.2}
\pounds(x)=\lim_{n\rightarrow \infty}\max_{ 0\leq
k<n}\left|2\right|^{k}
\zeta\left(\frac{x}{2^{k+1}},\frac{x}{2^{k}},\frac{x}{2^{k+1}}\right)
\end{equation}
exists and  $f:G\rightarrow X$ is a mapping
satisfying
\begin{eqnarray}\label{3.3}
&&\left\|f\left(\frac{x+y+z}{2}\right)+f\left(\frac{x-y+z}{2}\right)-f(x)-f(z)\right\|_X\le  \zeta(x,y,z)
\end{eqnarray}
Then the limit
$$
A(x):=\lim_{n\rightarrow
\infty}2^{n}f\left(\frac{
x}{2^{n}}\right)
$$
exists  for all $x\in G$ and defines an additive
mapping $A:G\rightarrow X$ such that
\begin{equation}\label{3.5}
\|f(x)-A(x)\|\leq \pounds(x).
\end{equation}
Moreover, if
$$
\lim_{j\rightarrow \infty}\lim_{n\rightarrow \infty} \max_{ j\leq
k<n+j}\left|2\right|^{k}
\zeta\left(\frac{x}{2^{k+1}},\frac{x}{2^{k}},\frac{x}{2^{k+1}}\right)=0\:,
$$
then $A$ is the unique additive mapping satisfying $(\ref{3.5})$.
\end{thm}
\begin{proof}
Setting $y=2x$ and $z=x$  in (\ref{3.3}), we get
\begin{equation}\label{3.7}
\left\|f(2x) -2f(x)
\right\|_Y \leq \zeta(x,2x,x)
\end{equation}
for all $x\in G$. Replacing $x$ by $\frac{x}{2^{n+1}}$ in
(\ref{3.7}), we obtain
\begin{eqnarray} \label{3.8}
&&\left\|2^{n+1}f\left(\frac{x}{2^{n+1}}\right)-
2^{n}f\left(\frac{x}{2^{n}}\right)\right\|\leq \left|2\right|^{n}
\zeta\left(\frac{x}{2^{n+1}},\frac{x}{2^{n}},\frac{x}{2^{n+1}}\right).
\end{eqnarray}
Thus, it follows from (\ref{3.1}) and (\ref{3.8}) that the
sequence $\left\{2^{n}f\left(\frac{x}{2^{n}}\right)\right\}_{n\geq 1}$ is a Cauchy sequence.
Since $X$ is complete, it follows that
$\left\{2^{n}f\left(\frac{x}{2^{n}}\right)\right\}_{n\geq 1}$ is a convergent sequence. Set
\begin{equation}\label{lim}
A(x):=\lim_{n\rightarrow \infty}2^{n}f\left(\frac{x}{2^{n}}\right).\end{equation}
 By induction on $n$, one can show that
\begin{eqnarray}\label{3.9}
&&\left\|2^{n}f\left(\frac{x}{2^{n}}\right)-f(x)\right\|\leq
\max\left\{\left|2\right|^{k}
\zeta\left(\frac{x}{2^{k+1}},\frac{x}{2^{k}},\frac{x}{2^{k+1}}\right);
0\leq k<n\right\}
\end{eqnarray}
for all $n\geq 1$ and  $x\in G$. By taking $n\to \infty$ in
(\ref{3.9}) and using (\ref{3.2}), one obtains (\ref{3.5}). By
(\ref{3.1}), (\ref{3.3}) and (\ref{lim}), we get
\begin{eqnarray*}
&& \left\|A\left(\frac{x+y+z}{2}\right)+A\left(\frac{x-y+z}{2}\right)-A(x)-A(z)\right\|
\\ &&  = \lim_{n\to\infty} |2|^{n}\left\|f\left(\frac{x+y+z}{2^{n+1}}\right)+f\left(\frac{x-y+z}{2^{n+1}}\right)-f\left(\frac{x}{2^{n}}\right)
-f\left(\frac{z}{2^{n}}\right)\right\|
\\ && \le  \lim_{n\to\infty} |2|^n \zeta\left(\frac{x}{2^n},\frac{y}{2^n},
\frac{z}{2^n}\right)=0
\end{eqnarray*}
 for all $x,y,z\in X$ . Thus
\begin{equation}\label{hhhf}
 A\left(\frac{x+y+z}{2}\right)+A\left(\frac{x-y+z}{2}\right)=A(x)+A(z)\end{equation}
for all $x,y,z\in G$. Letting $y=0$ in (\ref{hhhf}), we get
\begin{equation}\label{hhhfx}
 2L\left(\frac{x+z}{2}\right)=L(x)+L(z)\end{equation}
 for all $x,z\in G$. Since
 $$L(0)=\lim_{n\rightarrow +\infty}2^nf\left(\frac{0}{2^n}\right)=\lim_{n\rightarrow +\infty}2^nf(0)=0,$$
 by letting $y=2x$ and $z=x$ in (\ref{hhhf}), we get
 $$A(2x)=2A(x)$$
 for all $x\in G$. Replacing $x$ by $2x$ and $z$ by $2z$ in (\ref{hhhfx}), we obtain
 $$A(a+z)=A(x)+A(z)$$
 for all $x,z\in G$. Hence $A:G\rightarrow X$ is additive. \\
To prove the uniqueness property of $A$, let $L$ be another
mapping satisfying (\ref{3.5}). Then we have
\begin{eqnarray*}
&&  \Big\|A(x)-L(x)\Big\|_X \\ &&=\lim_{n\rightarrow
\infty}\left|2\right|^{n}\left\|A\left(\frac{
x}{2^{n}}\right)-L\left(\frac{
x}{2^{n}}\right)\right\|_X\\
&&\leq
  \lim_{k\rightarrow \infty}\left|2\right|^{n}
  \max\left\{\left\|A\left(\frac{
x}{2^{n}}\right)-f\left(\frac{
x}{2^{n}}\right)\right\|_X,
  \left\|f\left(\frac{
x}{2^{n}}\right)-L\left(\frac{
x}{2^{n}}\right)\right\|_X\right\}
  \\ && \leq \lim_{j\rightarrow \infty}\lim_{n\rightarrow \infty} \max_{ j\leq
k<n+j}\left|2\right|^{k}
\zeta\left(\frac{x}{2^{k+1}},\frac{x}{2^{k}},\frac{x}{2^{k+1}}\right)=0
\end{eqnarray*}
for all $x\in G$. Therefore, $A=L$. This completes the proof.
\end{proof}
\begin{cor}
Let $\xi:[0,\infty) \rightarrow [0,\infty)$ is a function
satisfying
\begin{equation*}
\xi\left(\frac{t}{|2|}\right)\leq
\xi\left(\frac{1}{|2|}\right)\xi(t),\quad
\xi\left(\frac{1}{|2|}\right)<
\frac{1}{|2|}
\end{equation*}
for all $t\geq 0$.  Assume that $\kappa >0$ and  $f: G \rightarrow
X$ is a mapping such that
\begin{eqnarray*}
\left\|f\left(\frac{x+y+z}{2}\right)+f\left(\frac{x-y+z}{2}\right)-f(x)-f(z)\right\|_Y\leq
\kappa\left(\xi(|x|)+\xi(|y|)+\xi(|z|)\right)
\end{eqnarray*}
for all $x,y,z\in G$. Then there exists a unique
 additive mapping \mbox{$A: G \rightarrow X$} such that
\begin{equation*}
\|f(x)-A(x)\|\leq \frac{(2+|2|)\xi(|x|)}{|2|}
\end{equation*}
\end{cor}
\begin{proof}
If we define $\zeta: G^{3} \rightarrow [0,\infty)$ by
$$\zeta(x,y,z):=\kappa\left(\xi(|x|)+\xi(|y|)+\xi(|z|)\right)\:,$$
then we have
$$
\lim_{n\rightarrow
\infty}\left|2\right|^{n}\zeta\left(\frac{
x}{2^{n}},\frac{y}{2^{n}},\frac{
z}{2^{n}}\right)=0\:,\ \ \text{for all $x,y,z\in G$. }
$$
On the other hand, it follows that
\begin{eqnarray*}
\pounds(x)=\zeta\left(\frac{x}{2},x,\frac{x}{2}\right)=\frac{(2+|2|)\xi(|x|)}{|2|}
\end{eqnarray*}
 exists for all $x\in G$. Also, we have
\begin{eqnarray*}
&&\lim_{j\rightarrow \infty}\lim_{n\rightarrow \infty} \max_{ j\leq
k<n+j}\left|2\right|^{k}
\zeta\left(\frac{x}{2^{k+1}},\frac{x}{2^{k}},\frac{x}{2^{k+1}}\right)=\lim_{j\rightarrow \infty} \left|2\right|^{j}
\zeta\left(\frac{x}{2^{j+1}},\frac{x}{2^{j}},\frac{x}{2^{j+1}}\right)=0.
\end{eqnarray*}
Thus, applying Theorem \ref{th3.1}, we have the conclusion. This
completes the proof.
\end{proof}

\begin{thm} \label{th3.2}
Let $G$ is an additive semigroup and that $X$ is a non-Archimedean
Banach space. Assume that $\zeta: G^{3} \rightarrow
[0,+\infty)$ is a function such that
\begin{equation}\label{3.16}
\lim_{n\rightarrow \infty}
\frac{\zeta\left(2^n x,2^n y,2^n z\right)}{2^n}=0
\end{equation}
for all $x,y,z\in G$. Suppose that, for any $x\in G$, the
limit
\begin{equation}\label{3.17}
\pounds(x)=\lim_{n\rightarrow \infty}\max_{0\leq
k<n}\frac{\zeta\left(2^k
x,2^{k+1} x,2^k x\right)}{|2|^{k}}
\end{equation}
exists and  $f:G\rightarrow X$ is a mapping with $f(0)=0$ and
satisfying (\ref{3.3}). Then the limit 
$$ A(x):=\lim_{n\rightarrow
\infty}\frac{f(2^n x)}{2^n}
$$ 
exists for all $x\in G$ and
\begin{equation}\label{3.19}
\|f(x)-A(x)\|\leq \frac{\pounds(x)}{|2|}.
\end{equation}
for all $x\in G$. Moreover, if
\begin{equation*}\label{3.20}
\lim_{j\rightarrow \infty}\lim_{n\rightarrow \infty}\max_{j\leq
k<n+j}\frac{\zeta\left(2^k
x,2^{k+1} x,2^k x\right)}{|2|^{k}}=0,
\end{equation*}
then $A$ is the unique mapping satisfying $(\ref{3.19})$.
\end{thm}
\begin{proof}
It follows from (\ref{3.7}), that
\begin{equation}\label{3.21}
\left\|f(x)-\frac{f(2x)}{2}\right\|_X\leq
\frac{\zeta(x,2x,x)}{|2|}
\end{equation}
for all $x\in G$. Replacing $x$ by $2^n x$ in
(\ref{3.21}), we obtain
\begin{eqnarray} \label{3.22}
&&\left\|\frac{f( 2^n x)}{2^n}-
\frac{f(2^{n+1}x)}{2^{n+1}}\right\|_X
\leq
\frac{\zeta\left(2^n
x,2^{n+1} x,2^n x\right)}{|2|^{n+1}}.
\end{eqnarray}
Thus it follows from (\ref{3.16}) and (\ref{3.22}) that the
sequence
$\left\{\frac{f( 2^n x)}{2^n}\right\}_{n\geq
1}$ is convergent. Set
$$A(x):=\lim_{n\rightarrow
\infty}\frac{f( 2^n x)}{2^n}.$$
On the other hand, it follows from (\ref{3.22}) that
\begin{eqnarray*}
\left\|\frac{f( 2^p x)}{2^p}-
\frac{f( 2^q x)}{2^q}\right\|=
\left\|\sum_{k=p}^{q-1}\frac{f( 2^{k+1}x)}{2^{k+1}}-
\frac{f( 2^k x)}{2^k}\right\|
&\leq&
 \max_{p\leq k< q}\left\{\left\|\frac{f( 2^{k+1}x)}{2^{k+1}}-
\frac{f( 2^k x)}{2^k}\right\|\right\}\\
&\leq&
 \frac{1}{|2|}\max_{p\leq k<q}\frac{\zeta\left(2^k
x,2^{k+1} x,2^k x\right)}{|2|^{k}}
\end{eqnarray*}
for all $x\in G$ and  $p,q\geq 0$ with $q>p\geq 0$. Letting $p=0$,
 taking $q\rightarrow \infty$ in the last inequality and using
(\ref{3.17}), we obtain (\ref{3.19}).

The rest of the proof is similar to the proof of Theorem
\ref{th3.1}. This completes the proof.
\end{proof}
\begin{cor}
Let $\xi:[0,\infty) \rightarrow [0,\infty)$ be a function
satisfying
\begin{equation*}
\xi\left(\left|2\right|t\right)\leq
\xi\left(\left|2\right|\right)\xi(t),\quad
\xi\left(\left|2\right|\right)<
\left|2\right|
\end{equation*}
for all $t\geq 0$. Let $\kappa >0$ and  $f: G \rightarrow X$ be a
mapping  satisfying
\begin{eqnarray*}
&&\left\|f\left(\frac{x+y+z}{2}\right)+f\left(\frac{x-y+z}{2}\right)-f(x)-f(z)\right\|\leq
\kappa\left(\xi(|x|)\cdot\xi(|y|)\cdot\xi(|z|)\right)
\end{eqnarray*}
for all $x,y,z\in G$. Then there exists a unique
 additive mapping \mbox{$A: G \rightarrow X$} such that
\begin{equation*}
\|f(x)-A(x)\|\leq \kappa \xi(|x|)^{3}.
\end{equation*}
\end{cor}
\begin{proof}
 If we define $\zeta: G^{3} \rightarrow
[0,\infty)$ by
$$\zeta(x,y,z):=\kappa\left(\xi(|x|)\cdot\xi(|y|)\cdot\xi(|z|)\right)$$
and apply Theorem \ref{th3.2}, then we get the conclusion.
\end{proof}

\section{\bf Random Stability of the Functional Equation
(\ref{main})}

 In this section, using the fixed point and direct methods, we prove the generalized
Hyers-Ulam-Rassias stability of the functional equation
(\ref{main}) in random normed spaces.

\subsection{Direct Method}

\begin{thm}\label{th4.1}
{\it Let $X$ be a real linear space, $(Z,\mu ',\min)$ be an
RN-space and $\varphi:X^3 \rightarrow Z$ be a function such that
there exists $0<\alpha<\frac{1}{2}$ such that
\begin{equation}\label{4.1}
\mu'_{\varphi(\frac{x}{2},\frac{y}{2},\frac{z}{2})}(t)\geq \mu'_{
\varphi(x,y,z)}\left(\frac{t}{\alpha}\right)\:,\ \ \text{for all $x,y,z\in X$ and $t>0$ }
\end{equation}
and
$$\lim_{n\rightarrow
\infty}\mu'_{\varphi(\frac{x}{2^n},\frac{y}{2^n},\frac{z}{2^n})}\left(\frac{t}{2^{n}}\right)=1\:,\ \ \text{for all $x,y,z\in X$ and $t>0$. }$$
Let $(Y,\mu,\min)$ be a complete
RN-space. If $f:X\rightarrow Y$ is a mapping such that
\begin{equation}\label{4.2}
\mu_{f\left(\frac{x+y+z}{2}\right)+f\left(\frac{x-y+z}{2}\right)-f(x)-f(z)}(t) \geq
\mu'_{\varphi(x,y,z)}(t)
\end{equation}
for all $x,y,z\in X$ and $t>0$,  then the limit
$$
A(x)=\lim_{n\rightarrow \infty }2^{n}f\left(\frac{x}{2^n}\right)
$$

exists for all $x\in X$ and defines a unique additive mapping $A:
X \rightarrow Y$ such that 
\begin{equation}\label{4.4}
\mu_{f(x)-A(x)}(t)\geq
\mu'_{\varphi(x,2x,x)}\left(\frac{(1-2\alpha)t}{\alpha}\right)\:,\ \text{for all $x\in X$ and $t>0$.}
\end{equation}
}
\end{thm}
\begin{proof}
Setting $y=2x$ and $z=x$ in (\ref{4.2}), we see that
\begin{equation}\label{b}
\mu_{f(2x)-2f(x)}(t)\geq \mu'_{\varphi(x,2x,x)}(t).
\end{equation}
Replacing $x$ by $\frac{x}{2}$ in (\ref{b}), we obtain
\begin{equation}\label{4.5}
\mu_{2f\left(\frac{x}{2}\right)-f(x)}(t)\geq
\mu'_{\varphi\left(\frac{x}{2},x,\frac{x}{2}\right)}(t)\geq \mu'_{\varphi(x,2x,x)}\left(\frac{t}{\alpha}\right)\:,\ \ \text{for all $x\in X$. }
\end{equation}
Replacing $x$ by $\frac{x}{2^n}$ in (\ref{4.5})
and using (\ref{4.1}), we obtain
\begin{equation*}
\mu_{2^{n+1}f\left(\frac{x}{2^{n+1}}\right) -2^n f\left(\frac{x}{2^{n}}\right)}(t)\geq
\mu'_{\varphi\left(\frac{x}{2^{n+1}},\frac{x}{2^{n}},\frac{x}{2^{n+1}}\right)}\left(\frac{
t}{2^n}\right)\geq \mu'_{\varphi(x,2x,x)}\left(\frac{ t}{2^n
\alpha^{n+1}}\right)
\end{equation*}
and therefore
\begin{eqnarray}\nonumber
\mu_{2^n f\left(\frac{x}{2^{n}}\right)-
f(x)}\left(\sum_{k=0}^{n-1}2^k\alpha^{k+1} t\right)
&=&\mu_{\sum_{k=0}^{n-1}2^{k+1} f\left(\frac{x}{2^{k+1}}\right)-2^k
f\left(\frac{x}{2^{k}}\right)}
\left(\sum_{k=0}^{n-1}2^k\alpha^{k+1} t\right)\\
 & \geq &
\nonumber T_{k=0}^{n-1}\left(\mu_{2^{k+1} f\left(\frac{x}{2^{k+1}}\right)-2^k
f\left(\frac{x}{2^{k}}\right)}(2^k \alpha ^{k+1} t)\right)\\
\nonumber &\geq&
T_{k=0}^{n-1}\Big(\mu'_{\varphi(x,2x,x)}(t)\Big)\\
\nonumber &=&\mu'_{\varphi(x,2x,x)}(t).
\end{eqnarray}
This implies that
\begin{equation}\label{4.7}
\mu_{2^n f\left(\frac{x}{2^{n}}\right)- f(x)}(t)\geq
\mu'_{\varphi(x,2x,x)}\left(\frac{t}{\sum_{k=0}^{n-1}2^k
\alpha^{k+1}}\right).
\end{equation}
Replacing $x$ by $\frac{x}{2^p}$ in (\ref{4.7}), we obtain
\begin{equation}\label{4.8}
\mu_{2^{n+p} f\left(\frac{x}{2^{n+p}}\right)- 2^p f\left(\frac{x}{2^{p}}\right)}(t) \geq
\mu'_{\varphi(x,2x,x)} \left(\frac{t}{\sum_{k=p}^{n+p-1}2^k
\alpha^{k+1}}\right).
\end{equation}
Since 
$$\lim_{p,n\rightarrow \infty}\mu'_{\varphi(x,2x,x)}
\left(\frac{t}{\sum_{k=p}^{n+p-1}2^k \alpha^{k+1}}\right)= 1,$$ 
it
follows that $\left\{2^n f(\frac{x}{2^n})\right\}_{n=1}^{\infty}$
is a Cauchy sequence in a complete RN-space $(Y,\mu,\min)$ and therefore
there exists a point $A(x)\in Y$ such that
$$\lim_{n\rightarrow \infty}2^n f\left(\frac{x}{2^n}\right)=A(x).$$ Fix $x\in X$
and set $p=0$ in (\ref{4.8}) and thus, for any $\epsilon>0$,
\begin{eqnarray}\label{4.10}\nonumber
 \mu_{A(x)-f(x)}(t+\epsilon) &\geq& T\left(\mu_{A(x)-
 2^n f\left(\frac{x}{2^n}\right)}(\epsilon), \mu_{2^n f\left(\frac{x}{2^n}\right)- f(x)}(t)\right)\\
  &\geq&
 T\left(\mu_{A(x)- 2^n f\left(\frac{x}{2^n}\right)}(\epsilon),\mu'_{\varphi(x,2x,x)}\left(\frac{t}{\sum_{k=0}^{n-1} 2^k
\alpha^{k+1}}\right) \right).
\end{eqnarray}
Taking  $n\rightarrow \infty$ in (\ref{4.10}), we obtain
\begin{equation}\label{4.11}
\mu_{A(x)-f(x)}(t+\epsilon)\geq
\mu'_{\varphi(x,2x,x)}\left(\frac{(1-2\alpha)t}{\alpha}\right).
\end{equation}
Since $\epsilon$ is arbitrary, by taking $\epsilon \rightarrow 0$
in (\ref{4.11}), we get
\begin{equation*}
\mu_{A(x)-f(x)}(t)\geq
\mu'_{\varphi(x,2x,x)}\left(\frac{(1-2\alpha)t}{\alpha}\right).
\end{equation*}
Replacing $x,y$ and $z$ by $\frac{x}{2^n}, \frac{y}{2^n}$ and $\frac{z}{2^n}$ in
(\ref{4.2}), respectively, we get
\begin{equation*}
\mu_{2^nf\left(\frac{x+y+z}{2^{n+1}}\right)+2^{n}f\left(\frac{x-y+z}{2^{n+1}}\right)-2^{n}f\left(\frac{x}{2^{n}}\right)-
2^{n}f\left(\frac{z}{2^{n}}\right)}(t)\geq
\mu'_{\varphi\left(\frac{x}{2^n},\frac{y}{2^n},\frac{z}{2^n}\right)}\left(\frac{t}{2^{n}}\right)
\end{equation*}
for all $x,y,z\in X$ and $t>0$. Since 
$$\lim_{n\rightarrow
\infty}\mu'_{\varphi\left(\frac{x}{2^n},\frac{y}{2^n},\frac{z}{2^n}\right)}\left(\frac{t}{2^{n}}\right)
=1\:,$$
we conclude that $A$ satisfies (\ref{main}). On the other hand
\begin{eqnarray*}
2A\left(\frac{x}{2}\right)-A(x)=\lim_{n\rightarrow \infty}2^{n+1}
f\left(\frac{x}{2^{n+1}}\right)-\lim_{n\rightarrow \infty}2^{n}
f\left(\frac{x}{2^{n}}\right)=0.
\end{eqnarray*}
This implies that $A: X \rightarrow Y$ is an additive mapping. To prove the uniqueness of the additive mapping $A$, assume that
there exists another additive mapping $L: X \rightarrow Y$ which
satisfies (\ref{4.4}). Then we have
\begin{eqnarray*}
\mu_{A(x)-L(x)}(t)&=&\lim_{n\rightarrow \infty}\mu_{2^n
A\left(\frac{x}{2^n}\right)- 2^n L\left(\frac{x}{2^n}\right)}(t)
\\ \nonumber &\geq& \lim_{n\rightarrow
\infty}\min\left\{\mu_{2^n A\left(\frac{x}{2^n}\right)- 2^n
f\left(\frac{x}{2^n}\right)}\left (\frac{t}{2}\right ),\mu_{2^n
f\left(\frac{x}{2^n}\right)
- 2^n L\left(\frac{x}{2^n}\right)}\left (\frac{t}{2}\right )\right \}\\
\nonumber &\geq&
 \lim_{n\rightarrow \infty}\mu'_{\varphi\left(\frac{x}{2^n},\frac{2x}{2^n},\frac{x}{2^n}\right)}\left(\frac{(1-2 \alpha)t}{2^{n}}\right) \geq \lim_{n\rightarrow \infty}\mu'_{\varphi(x,2x,x)}\left(\frac{(1-2 \alpha)t}{2^{n}\alpha^n }\right).
\end{eqnarray*}
Since $$\lim_{n\rightarrow
\infty}\mu'_{\varphi(x,2x,x)}\left(\frac{(1-2 \alpha)t}{2^{n}\alpha^n }\right)=1\:,$$
 it follows that $$\mu_{A(x)- L(x)}(t)=1\:,\ \ \text{for
all $t>0$}$$ and so $A(x)=L(x)$. This completes the proof.
\end{proof}
\begin{cor}
{\it Let $X$ be a real normed linear space, $(Z,\mu',\min)$ be an
RN-space and $(Y,\mu,\min)$ be a complete RN-space. Let $r$ be a
positive real number with $r>1$ , $z_0\in Z$ and $f:X\rightarrow
Y$ be a mapping satisfying
\begin{equation}\label{jj}
\mu_{f\left(\frac{x+y+z}{2}\right)+f\left(\frac{x-y+z}{2}\right)-f(x)-f(z)}(t) \geq
\mu'_{(\|x\|^r+\|y\|^r+\|z\|^r)z_0}(t)
\end{equation}
for all $x,y\in X$ and $t>0$.  Then the limit
$$
A(x)=\lim_{n\rightarrow \infty }2^n f\left(\frac{x}{2^n}\right)
$$
exists for all $x\in X$ and defines a unique additive mapping $A:
X \rightarrow Y$ such that and
\begin{equation*}
\mu_{f(x)- A(x)}(t)\geq
\mu'_{\|x\|^{p}z_0}\left(\frac{(2^r-2)t}{2^r +2}\right)\:,\ \text{for all $x\in X$ and $t>0$.}
\end{equation*}
}
\end{cor}
\begin{proof}
Let $\alpha=2^{-r}$ and $\varphi:X^3 \rightarrow Z$ be a mapping
defined by $$\varphi(x,y,z)=(\|x\|^r+ \|y\|^r+\|z\|^r)z_0\:.$$ Then, from Theorem
\ref{th4.1}, the conclusion follows.
\end{proof}
\begin{thm}\label{th4.1q}
{\it Let $X$ be a real linear space, $(Z,\mu ',\min)$ be an
RN-space and $\varphi:X^3 \rightarrow Z$ be a function such that
there exists $0<\alpha<2$ such that
$$
\mu'_{\varphi(2x,2y,2z)}(t)\geq \mu'_{\alpha \varphi(x,y,z)}(t)\:,\ \text{for all $x\in X$ and $t>0$}
$$
 and
$$\lim_{n\rightarrow
\infty}\mu'_{\varphi(2^n x,2^n y,2^n z)}(2^n x)=1$$ for all $x,y,z\in X$ and
$t>0$. Let $(Y,\mu,\min)$ be a complete RN-space. If
$f:X\rightarrow Y$ is a mapping satisfying (\ref{4.2}), then the
limit
$$
A(x)=\lim_{n\rightarrow \infty }\frac{f(2^n x)}{2^n}
$$
exists for all $x\in X$ and defines a unique additive mapping $A:
X \rightarrow Y$ such that 
\begin{equation}\label{4.4q}
\mu_{f(x)-A(x)}(t)\geq \mu'_{\varphi(x,2x,x)}((2-\alpha) t)\:,\ \text{for all $x\in X$ and $t>0$.}
\end{equation}
}
\end{thm}
\begin{proof}
It follows from  (\ref{b}) that
\begin{equation}\label{bq}
\mu_{\frac{f(2x)}{2}-f(x)}(t)\geq \mu'_{\varphi(x,2x,x)}(2t).
\end{equation}
Replacing $x$ by $2^n x$ in (\ref{bq}), we obtain that
\begin{equation*}
\mu_{\frac{f(2^{n+1}x)}{2^{n+1}}-\frac{f(2^{n}x)}{2^{n}}}(t)\geq
\mu'_{\varphi(2^n x,2^{n+1} x,2^n x)}(2^{n+1}t)\geq
\mu_{\varphi(x,2x,x)}\left(\frac{2^{n+1}t}{\alpha^n}\right).
\end{equation*}
The rest of the proof is similar to the proof of Theorem
\ref{th4.1}.
\end{proof}
\begin{cor}
{\it Let $X$ be a real normed linear space, $(Z,\mu',\min)$ be an
RN-space and $(Y,\mu,\min)$ be a complete RN-space. Let $r$ be a
positive real number with $0<r<1$ , $z_0\in Z$ and $f:X\rightarrow
Y$ be a mapping satisfying (\ref{jj}).  Then the limit
$$A(x)=\lim_{n\rightarrow \infty}\frac{f(2^n x)}{2^n}$$ exists for
all $x\in X$ and defines a unique additive mapping $A: X
\rightarrow Y$ such that and
\begin{equation*}
\mu_{f(x)- A(x)}(t)\geq
\mu'_{\|x\|^{p}z_0}\left(\frac{(2-2^r)t}{2^r +2}\right)
\end{equation*}
for all $x\in X$ and $t>0$.}
\end{cor}
\begin{proof}
Let $\alpha=2^{r}$ and $\varphi:X^3 \rightarrow Z$ be a mapping
defined by $$\varphi(x,y,z)=(\|x\|^r+ \|y\|^r+\|z\|^r)z_0\:.$$ Then, from Theorem
\ref{th4.1q}, the conclusion follows.
\end{proof}

\subsection{\bf Fixed Point Method}

\begin{thm}\label{th4.2}
{\it Let $X$ be a linear space, $(Y,\mu,T_M)$ be a complete
RN-space and $\Phi$ be a mapping from $X^3$ to $D^+$
$($$\Phi(x,y,z)$ is denoted by $\Phi_{x,y.z}$$)$ such that there
exists  $0<\alpha<\frac{1}{2}$ satisfying the property 
\begin{equation}\label{4.20}
\Phi_{2x,2y,2z}(t)\leq \Phi_{x,y,z}(\alpha t)
\end{equation}
for all $x,y,z\in X$ and $t>0$. Let $f:X\rightarrow Y$ be a mapping
satisfying
\begin{equation}\label{4.21}
\mu_{f\left(\frac{x+y+z}{2}\right)+f\left(\frac{x-y+z}{2}\right)-f(x)-f(z)}(t)\geq \Phi_{x,y,z}(t)
\end{equation}
for all $x,y,z\in X$ and $t>0$. Then, for all $x\in X$,
$$
A(x):=\lim_{n\rightarrow \infty} 2^n f\left(\frac{x}{2^n}\right)
$$
exists  and $A:X \rightarrow Y$ is a unique additive mapping such
that
\begin{equation}\label{4.23}
\mu_{f(x)-A(x)}(t)\geq
\Phi_{x,2x,x}\left(\frac{(1-2\alpha)t}{\alpha}\right)
\end{equation}
for all $x\in X$ and $t>0$.}
\end{thm}
\begin{proof}
Setting $y=2x$ and $z=x$ in (\ref{4.21}), we have
\begin{equation}\label{4.24}
\mu_{2f(\frac{x}{2})-f(x)}(t)\geq
\Phi_{\frac{x}{2},x,\frac{x}{2}}(t)\geq \Phi_{x,2x,x}\left(\frac{t}{\alpha}\right)
\end{equation}
for all $x\in X$ and $t>0$. Consider the set $S:=\{g:X\rightarrow Y\}
$
and the generalized metric $d$ in $S$ defined by
\begin{equation}
d(f,g)=\inf_{u\in (0,\infty)}\left \{\mu_{g(x)-h(x)}(ut)\geq
\Phi_{x,2x,x}(t),\, \forall x\in X,\, t>0\right \},
\end{equation}
where $\inf\,\emptyset=+\infty$. It is easy to show that $(S,d)$
is complete (see \cite{mr}, Lemma 2.1). Now, we consider a linear mapping $J:(S,d)\rightarrow (S,d)$ such
that
\begin{equation}
Jh(x):=2h\left(\frac{x}{2}\right)
\end{equation}
for all $x\in X$. First, we prove that $J$ is a strictly contractive mapping with
the Lipschitz constant $2\alpha$. In fact, let $g,h\in S$ be  such
that $d(g,h)<\epsilon$. Then we have
$$
\mu_{g(x)-h(x)}(\epsilon t)\geq \Phi_{x,2x,x}(t)\:,\ \text{for all $x\in X$ and $t>0$}
$$
 and so
\begin{eqnarray*}
\mu_{Jg(x)-Jh(x)}(2\alpha \epsilon t)=
\mu_{2g(\frac{x}{2})-2h(\frac{x}{2})}(2\alpha \epsilon t) &=&
 \mu_{g(\frac{x}{2})-h(\frac{x}{2})} (\alpha \epsilon t)
\\ &\geq& \Phi_{\frac{x}{2},x,\frac{x}{2}}(\alpha t) \\ &\geq& \Phi_{x,2x,x}(t)
\end{eqnarray*}
for all $x\in X$ and  $t>0$. Thus $d(g,h)<\epsilon$ implies that
$d(Jg,Jh)<2\alpha \epsilon$. This means that
$
d(Jg,Jh)\leq 2\alpha d(g,h)
$
for all $g,h\in S$. It follows from (\ref{4.24}) that
\begin{equation*}
d(f,Jf)\leq \alpha.
\end{equation*}
By Theorem \ref{thras}, there exists a mapping $A:X\rightarrow Y$
satisfying the following:\\
 (1) $A$ is a fixed point of $J$, that
is,
\begin{equation}\label{4.31q}
A\Big(\frac{x}{2}\Big)=\frac{1}{2}A(x)
\end{equation}
for all $x\in X$. The mapping $A$ is a unique fixed point of $J$
in the set
$$
\Omega=\{h\in S: d(g,h)<\infty\}.
$$
This implies that $A$ is a unique mapping satisfying (\ref{4.31q})
such that there exists $u\in (0,\infty)$ satisfying
$$
\mu_{f(x)-A(x)}(ut)\geq \Phi_{x,2x,x}(t)\:,\ \text{for all $x\in X$ and $t>0$.}
$$
 (2) $d(J^n f, A)\rightarrow 0$ as
$n\rightarrow \infty$. This implies the equality
\begin{equation*}
\lim_{n\rightarrow \infty}2^n f\Big(\frac{x}{2^n}\Big)= A(x)
\end{equation*}
for all $x\in X$.\\
(3) $d(f,A)\leq \frac{d(f,Jf)}{1-2\alpha}$ with $f\in \Omega$,
which implies the inequality
$$
d(f,A)\leq \frac{\alpha}{1-2\alpha}
$$
and so
\begin{equation*}
\mu_{f(x)-A(x)}\left(\frac{\alpha t}{1-2\alpha}\right)\geq
\Phi_{x,2x,x}(t)
\end{equation*}
for all $x\in X$ and $t>0$. This implies that the inequality
(\ref{4.23}) holds. On the other hand
\begin{equation*}
\mu_{2^nf\left(\frac{x+y+z}{2^{n+1}}\right)+2^{n}f\left(\frac{x-y+z}{2^{n+1}}\right)-2^{n}f\left(\frac{x}{2^{n}}\right)-
2^{n}f\left(\frac{z}{2^{n}}\right)}(t)\geq
\Phi_{\frac{x}{2^n},\frac{y}{2^n},\frac{z}{2^n}}\left(\frac{t}{2^{n}}\right)
\end{equation*}
for all $x,y,z\in X$, $t>0$ and $n\geq 1$. By (\ref{4.20}),
we know that
$$
\Phi_{\frac{x}{2^n},\frac{y}{2^n},\frac{z}{2^n}}\left(\frac{t}{2^{n}}\right)\geq
\Phi_{x,y,z}\left(\frac{t}{(2\alpha)^{n}}\right)
.$$
Since
$$\lim_{n\rightarrow\infty}\Phi_{x,y,z}\left(\frac{t}{(2\alpha)^{n}}\right)=1\:,\ \text{for all $x,y,z\in X$ and $t>0$,}$$ 
we have
$$\mu_{A\left(\frac{x+y+z}{2}\right)+A\left(\frac{x-y+z}{2}\right)-A(x)-A(z)}(t)=1$$ for all $x,y,z\in
X$ and $t>0$. Thus the mapping $A: X\rightarrow Y$ satisfies 
(\ref{main}). Furthermore
\begin{eqnarray*}
A(2x)-2A(x)&=&\lim_{n\rightarrow \infty}2^{n}
f\left(\frac{x}{2^{n-1}}\right)-2\lim_{n\rightarrow \infty}2^{n}
f\left(\frac{x}{2^{n}}\right)\\ &=&2\left[\lim_{n\rightarrow
\infty}2^{n-1} f\left(\frac{x}{2^{n-1}}\right)-\lim_{n\rightarrow
\infty}2^{n} f\left(\frac{x}{2^{n}}\right)\right]\\ &=&0.
\end{eqnarray*}

This completes the proof.
\end{proof}
\begin{cor}
{\it Let $X$ be a real normed space, $\theta\geq0$ and $r$ be a
real number with $r>1$. Let $f: X\rightarrow Y$ be a mapping
satisfying
\begin{equation}\label{cor1}
\mu_{f\left(\frac{x+y+z}{2}\right)+f\left(\frac{x-y+z}{2}\right)-f(x)-f(z)}(t)\geq
\frac{t}{t+\theta\big(\|x\|^r+\|y\|^r+\|z\|^r\big)}
\end{equation}
for all $x,y,z\in X$ and $t>0$. Then
$$
A(x)=\lim_{n\rightarrow\infty}2^n f\left(\frac{x}{2^n}\right)
$$
exists  for all $x\in X$ and $A: X\rightarrow Y$ is a unique
additive mapping such that
\begin{equation*}
\mu_{f(x)-A(x)}(t)\geq \frac{(2^r -2)t}{(2^r -2)t+(2^r +2) \theta\|x\|^r}
\end{equation*}
for all $x\in X$ and $t>0$.}
\end{cor}
\begin{proof}
The proof follows from Theorem \ref{th4.2} if we take
\begin{equation*}
\Phi_{x,y,z}(t)=\frac{t}{t+\theta\big(\|x\|^r+\|y\|^r+\|z\|^r\big)}
\end{equation*}
for all $x,y,z\in X$ and $t>0$. In fact, if we choose
$\alpha=2^{-r}$, then we get the desired result.
\end{proof}
\begin{thm}\label{th4.3}
{\it Let $X$ be a linear space, $(Y,\mu,T_M)$ be a complete
RN-space and $\Phi$ be a mapping from $X^3$ to $D^+$ ($\Phi(x,y,z)$
is denoted by $\Phi_{x,y,z})$ such that for some $0<\alpha<2$ it holds:
\begin{equation*}
\Phi_{\frac{x}{2},\frac{y}{2},\frac{z}{2}}(t)\leq \Phi_{x,y,z}(\alpha t)
\end{equation*}
for all $x,y,z\in X$ and $t>0$. Let $f:X\rightarrow Y$ be a mapping
satisfying (\ref{4.21}). Then the limit
$$
A(x):=\lim_{n\rightarrow \infty} \frac{f(2^n x)}{2^n}
$$
exists  for all $x\in X$ and $A:X \rightarrow Y$ is a unique
additive mapping such that
\begin{equation}\label{hh}
\mu_{f(x)-A(x)}(t)\geq \Phi_{x,2x,x}((2-\alpha)t)
\end{equation}
for all $x\in X$ and $t>0$.}
\end{thm}
\begin{proof}
Setting $y=2x$ and $z=x$ in (\ref{4.21}), we have
\begin{equation}\label{4.24q}
\mu_{\frac{f(2x)}{2}-f(x)}(t)\geq \Phi_{x,2x,x}(2t)
\end{equation}
for all $x\in X$ and $t>0$. Let $(S,d)$ be the generalized metric space defined in the proof of Theorem \ref{th4.1}. Now, we consider a linear mapping $J:(S,d)\rightarrow (S,d)$ such
that
\begin{equation}
Jh(x):=\frac{1}{2}h(2x)\:,\ \ \text{for all $x\in X$. }
\end{equation}
It follows from (\ref{4.24q}) that
$
d(f,Jf)\leq \frac{1}{2}.
$
By Theorem \ref{thras}, there exists a mapping $A:X\rightarrow Y$
satisfying the following:\\
 (1) $A$ is a fixed point of $J$, that
is,
\begin{equation}\label{4.31}
A(2x)=2A(x)\:,\ \ \text{for all $x\in X$.}
\end{equation}
 The mapping $A$ is a unique fixed point of $J$
in the set
$$
\Omega=\{h\in S: d(g,h)<\infty\}.
$$
This implies that $A$ is a unique mapping satisfying (\ref{4.31})
such that there exists $u\in (0,\infty)$ satisfying
$$
\mu_{f(x)-A(x)}(ut)\geq \Phi_{x,2x,x}(t)\:,\ \ \text{for all $x\in X$ and $t>0$.}
$$
 (2) $d(J^n f, A)\rightarrow 0$ as
$n\rightarrow \infty$. This implies the equality
\begin{equation*}
\lim_{n\rightarrow \infty}\frac{f(2^n x)}{2^n}= A(x)\:,\ \ \text{for all $x\in X$.}
\end{equation*}
(3) $d(f,A)\leq \frac{d(f,Jf)}{1-\frac{\alpha}{2}}$ with $f\in
\Omega$, which implies the inequality
\begin{equation*}
\mu_{f(x)-A(x)}\left(\frac{ t}{2-\alpha}\right)\geq \Phi_{x,2x,x}(t)\:,\ \text{for all $x\in X$ and $t>0$.}
\end{equation*}
This implies that the inequality
(\ref{hh}) holds. The rest of the proof is similar to the proof of
Theorem \ref{th4.2}.
\end{proof}
\begin{cor}
{\it Let $X$ be a real normed space, $\theta\geq0$ and $r$ be a
real number with $0<r<1$. Let $f: X\rightarrow Y$ be a mapping
satisfying (\ref{cor1}). Then the limit 
$$
A(x)=\lim_{n\rightarrow\infty}\frac{ f(2^n x) }{2^n}$$ exists  for
all $x\in X$ and $A: X\rightarrow Y$ is a unique additive mapping
such that
\begin{equation*}
\mu_{f(x)-A(x)}(t)\geq \frac{(2-2^r)t}{(2-2^r)t+(2^r +2)\theta\|x\|^r}\:,\ \text{for all $x\in X$ and $t>0$.}
\end{equation*}
}
\end{cor}
\begin{proof}
The proof follows from Theorem \ref{th4.3} if we take
\begin{equation*}
\Phi_{x,y}(t)=\frac{t}{t+\theta(\|x\|^r+\|y\|^r+\|z\|^r)}
\end{equation*}
for all $x,y,z\in X$ and $t>0$. In fact, if we choose
$\alpha=2^{r}$, then we get the desired result.
\end{proof}

\section{\bf Fuzzy Stability of the Functional Equation (\ref{main})}
Throughout this section, using the fixed point and direct methods, we prove the generalized Hyers-Ulam-Rassias stability of the functional
equation (\ref{main}) in fuzzy normed spaces.

\subsection{A direct method}
\begin{equation*}\end{equation*}
In this section, using the direct method, we prove the
Hyers-Ulam-Rassias stability of the functional equation (\ref{main}) in
fuzzy Banach spaces. Throughout this section, we assume that $X$
is a linear space, $(Y,N)$ is a fuzzy Banach space and $(Z,N')$ is
a fuzzy normed space. Moreover, we assume that $N(x,.)$ is a left
continuous function on $\mathbb R$.
\begin{thm}\label{th41}
Assume that a mapping $f:X\rightarrow Y$ satisfies the inequality
\begin{eqnarray}\nonumber \label{18}
N\left(f\left(\frac{x+y+z}{2}\right)+f\left(\frac{x-y+z}{2}\right)-f(x)-f(z),t \right)\\
\geq N'(\varphi(x,y,z),t)
\end{eqnarray}
for all $x,y,z\in X$, $t>0$ and
$\varphi:X^3\rightarrow Z$ is a mapping for which there is a
constant $r\in \mathbb R$ satisfying $0<|r|<\frac{1}{2}$ with
\begin{equation}\label{19cx}
N'\left(\varphi\left(\frac{x}{2},\frac{y}{2},\frac{z}
{2}\right),t\right)\geq
N'\left(\varphi(x,y,z),\frac{t}{|r|}\right)
\end{equation}
for all $x,y,z\in X$ and all $t>0$. Then we can find a
unique additive mapping $A:X\rightarrow Y$ satisfying (\ref{main})
and the inequality
\begin{equation}\label{20}
N(f(x)-A(x),t)\geq
N'\left(\frac{|r|\varphi(x,2x,x)}{1-2|r|},t
\right)\:,
\end{equation}
for all $x\in X$ and all $t>0$.
\end{thm}
\begin{proof}
It follows from (\ref{19cx}) that
\begin{eqnarray}
N'\left(\varphi\left(\frac{x}{2^{j}},\frac{y}{2^{j}},\frac{z}
{2^{j}}\right),t\right)\geq  N'\left(\varphi(x,y,z),\frac{t}{|r|^{j}}\right).
\end{eqnarray}
Therefore
$
N'\left(\varphi\left(\frac{x}{2^{j}},\frac{y}{2^{j}},\frac{z}
{2^{j}}\right),|r|^{j}t\right)\geq
N'\left(\varphi(x,y,z),t\right)
$
 for all $x,y,z\in X$ and all $t>0$. \\
Substituting $y=2x$ and $z=x$ in (\ref{18}), we obtain
\begin{equation}\label{uuu}
N\left(f(2x)-2f(x),t\right)\geq
N'(\varphi(x,2x,x),t)
\end{equation}
Thus
\begin{eqnarray}\label{22q}
N\left(f(x)-2f\left(\frac{x}{2}\right),t\right)
\geq N'\left(\varphi\left(\frac{x }{2},x,\frac{x
}{2}\right),t\right)
\end{eqnarray}
for all $x\in X$ and all $t>0$. Replacing $x$ by
$\frac{x}{2^{j}}$ in (\ref{22q}), we have
\begin{eqnarray}\nonumber\label{23}
N\left(2^{j+1}f\left(\frac{x}{2^{j+1}}\right)-2^{j}f\left(\frac{x}{2^{j}}\right)
,2^{j}t\right)
&\geq& N'\left(\varphi\left(\frac{x }{2^{j+1}},\frac{x
}{2^{j}},\frac{x}{2^{j+1}}\right),t\right)\\
 &\geq&
N'\left(\varphi\left(x,2x,x\right),\frac{t}{|r|^{j+1}}\right)
\end{eqnarray}
for all $x\in X$,  all $t>0$ and any integer $j\geq 0$. Hence
\begin{eqnarray}\nonumber
N\left(f(x)-2^{n}f\left(\frac{x}{2^{n}}\right),\sum_{j=0}^{n-1}
2^{j}|r|^{j+1}t\right) &=&
N\left(\sum_{j=0}^{n-1}\left[2^{j+1}f\left(\frac{x}{2^{j+1}}\right)-
2^{j}f\left(\frac{x}{2^{j}}\right)\right],\sum_{j=0}^{n-1}
2^{j}|r|^{j+1}t\right)\\ \nonumber &\geq&
\min_{0\leq j\leq
n-1}\left\{N\left(2^{j+1}f\left(\frac{x}{2^{j+1}}\right)-
2^{j}f\left(\frac{x}{2^{j}}\right),
2^{j}|r|^{j+1}t\right)\right\}\\ \nonumber &\geq& N'(\varphi(x,2x,x),t)\:,
\end{eqnarray}
which yields
\begin{eqnarray}\nonumber
N\left(2^{n+p}f\left(\frac{x}{2^{n+p}}\right)
-2^{p}f\left(\frac{x}{2^{p}}\right),\sum_{j=0}^{n-1}
2^{j}|r|^{j+1}t\right) &\geq&
N'\left(\varphi\left(\frac{x}{2^{p}},\frac{2x}{2^{p}},\frac{x}{2^{p}}
\right),t\right)\\ \nonumber &\geq&
N'\left(\varphi(x,2x,x),\frac{t}{|r|^{p}}\right)
\end{eqnarray}
for all $x\in X$, $t>0$ and any integers $n>0$, $p\geq 0$. Therefore
\begin{eqnarray*}
N\left(2^{n+p}f\left(\frac{x}{2^{n+p}}\right)
-2^{p}f\left(\frac{x}{2^{p}}\right),\sum_{j=0}^{n-1}
2^{j+p}|r|^{j+p+1}t\right) \geq
N'(\varphi(x,2x,x),t)
\end{eqnarray*}
for all $x\in X$, $t>0$ and any integers $n>0$, $p\geq 0$. Hence
one obtains
\begin{eqnarray}\label{lll}
&&N\left(2^{n+p}f\left(\frac{x}{2^{n+p}}\right)
-2^{p}f\left(\frac{x}{2^{p}}\right),t\right) \\ \nonumber
&&\geq
N'\left(\varphi(x,2x,x),\frac{t}{\sum_{j=0}^{n-1}
2^{j+p}|r|^{j+p+1}}\right)
\end{eqnarray}
for all $x\in X$, $t>0$ and any integers $n>0$, $p\geq 0$. Since,
the series
$$\sum_{j=0}^{\infty} 2^{j}|r|^{j}$$ is convergent series, we
see by taking the limit as $p\rightarrow \infty$ in the last
inequality that a sequence
$\Big\{2^{n}f\left(\frac{x}{2^{n}}\right)\Big\}$ is a
Cauchy sequence in the fuzzy Banach space $(Y,N)$ and thus it
converges in $Y$. Therefore a mapping $A:X\rightarrow Y$ defined by
$$A(x):=N-\lim_{n\rightarrow
\infty}2^{n}f\left(\frac{x}{2^{n}}\right)$$ is well defined
for all $x\in X$. This means that
\begin{equation}\label{gg}
\lim_{n\rightarrow
\infty}N\left(A(x)-2^{n}f\left(\frac{x}{2^{n}}\right)
,t\right)=1
\end{equation}
 for all $x\in X$ and all $t>0$. In addition, it
follows from (\ref{lll}) that
\begin{eqnarray*}
N\left(2^{n}f\left(\frac{x}{2^{n}}\right)-f(x),t\right)\geq
N'\left(\varphi(x,2x,x),\frac{t}{\sum_{j=0}^{n-1}
2^{j}|r|^{j+1}}\right)
\end{eqnarray*}
for all $x\in X$ and all $t>0$. Thus
\begin{eqnarray*}
&&N(f(x)-A(x),t)\\  &&\geq \min\left\{N\left(f(x)
-2^{n}f\left(\frac{x}{2^{n}}\right),(1-\epsilon)t\right),
N\left(A(x) -2^{n}f\left(\frac{x}{2^{n}}\right),\epsilon
t\right)\right\}\\ &&\geq
N'\left(\varphi(x,2x,x),\frac{t}{\sum_{j=0}^{n-1}
2^{j}|r|^{j+1}}\right)\geq
N'\left(\varphi(x,2x,x),\frac{(1-2|r|)\epsilon
t}{|r|} \right)
\end{eqnarray*}
for sufficiently large $n$ and for all $x\in X$, $t>0$ and
$\epsilon$ with $0<\epsilon<1$. Since $\epsilon$ is arbitrary and
$N'$ is left continuous, we obtain
\begin{eqnarray*}
N(f(x)-A(x),t)\geq
N'\left(\varphi(x,2x,x),\frac{(1-2|r|)
t}{|r|} \right)
\end{eqnarray*}
for all $x\in X$ and $t>0$. It follows from (\ref{18}) that
\begin{eqnarray*}
&&N\left(2^{n}f\left(\frac{x+y+z}{2^{n+1}}\right)+2^n f\left(\frac{x-y+z}{2^{n+1}}\right)-2^n f\left(\frac{x}{2^{n}}\right)-2^n f\left(\frac{z}{2^{n}}\right),t \right)\\
&&\geq
N'\left(\varphi\left(\frac{x}{2^n},\frac{y}{2^n},\frac{z}
{2^n}\right),\frac{t}{2^n}\right)\geq
N'\left(\varphi(x,y,z),\frac{t}{2^n|r|^{n}}\right)
\end{eqnarray*}
 for all $x,y,z\in X$, $t>0$ and all $n\in
 \mathbb{N}$. Since
 $$\lim_{n\rightarrow \infty}N'\left(\varphi(x,y,z),\frac{t}{2^n|r|^{n}}\right)=1$$
 and thus
 \begin{eqnarray*}
N\left(2^{n}f\left(\frac{x+y+z}{2^{n+1}}\right)+2^n f\left(\frac{x-y+z}{2^{n+1}}\right)-2^n f\left(\frac{x}{2^{n}}\right)-2^n f\left(\frac{z}{2^{n}}\right),t \right)\rightarrow 1
\end{eqnarray*}
 for all $x,y,z\in X$ and all $t>0$, we
 obtain in view of (\ref{gg})
 \begin{eqnarray*}
&&N\left(A\left(\frac{x+y+z}{2}\right)+A\left(\frac{x-y+z}{2}\right)-A(x)-A(z),t \right)\\
&&\geq \min\biggr\{N\biggr(A\left(\frac{x+y+z}{2}\right)+A\left(\frac{x-y+z}{2}\right)-A(x)-A(z)\\
&&-2^{n}f\left(\frac{x+y+z}{2^{n+1}}\right)+2^n f\left(\frac{x-y+z}{2^{n+1}}\right)-2^n f\left(\frac{x}{2^{n}}\right)-2^n f\left(\frac{z}{2^{n}}\right),\frac{t}{2}\biggr)\\ &&,
N\biggr(2^{n}f\left(\frac{x+y+z}{2^{n+1}}\right)+2^n f\left(\frac{x-y+z}{2^{n+1}}\right)-2^n f\left(\frac{x}{2^{n}}\right)-2^n f\left(\frac{z}{2^{n}}\right),\frac{t}{2} \biggr)\biggr\}\\ &&=
N\biggr(2^{n}f\left(\frac{x+y+z}{2^{n+1}}\right)+2^n f\left(\frac{x-y+z}{2^{n+1}}\right)-2^n f\left(\frac{x}{2^{n}}\right)-2^n f\left(\frac{z}{2^{n}}\right),\frac{t}{2} \biggr)\\ &&\geq
N'\left(\varphi(x,y,z),\frac{t}{2^{n+1}|r|^{n}}\right)\rightarrow
1~~~\mbox{as $n\rightarrow \infty$}\:,
\end{eqnarray*}
which implies
$$A\left(\frac{x+y+z}{2}\right)+A\left(\frac{x-y+z}{2}\right)=A(x)+A(z)$$ for all
$x,y,z\in X$. Thus $A:X\rightarrow Y$ is a mapping
satisfying the equation (\ref{main}) and the inequality (\ref{20}).\\
To prove the uniqueness, let us assume that there is another mapping
$L:X\rightarrow Y$ which satisfies the inequality (\ref{20}).
Since 
$$L(2^{n}x)=2^{n}L(x)\:,\ \  \text{for all $x\in X$,}$$ 
we have
\begin{eqnarray*}
&&N(A(x)-L(x),t)\\
&&=N\left(2^{n}A\left(\frac{x}{2^{n}}\right)-
2^{n}L\left(\frac{x}{2^{n}}\right),t\right)\\ &&\geq
\min\biggr\{N\left(2^{n}A\left(\frac{x}{2^{n}}\right)-
2^{n}f\left(\frac{x}{2^{n}}\right),\frac{t}{2}\right), N\left(2^{n}f\left(\frac{x}{2^{n}}\right)-
2^{n}L\left(\frac{x}{2^{n}}\right),\frac{t}{2}\right)\biggr\}\\
&&\geq
N'\left(\varphi\left(\frac{x}{2^n},\frac{2x}{2^n},\frac{x}
{2^n}\right),\frac{(1-2|r|)t}{|r|2^{n+1}}\right)\geq
N\left(\varphi(x,2x,x),\frac{(1-2|r|)t}{|r|^{n+1}2^{n+1}}\right)
\rightarrow 1 ~~\mbox{as $n\rightarrow \infty$}
\end{eqnarray*}
for all $t>0$. Therefore $A(x)=L(x)$ for all $x\in X$, this
completes the proof.
\end{proof}
\begin{cor}
Let $X$ be a normed space and that $(\mathbb R,N')$ is a fuzzy
Banach space. Assume that there exist a real number $\theta\geq 0$
and \mbox{$0<p<2$} such that a mapping $f:X\rightarrow Y$ with $f(0)=0$
satisfies the following inequality
\begin{eqnarray*}\nonumber
N\left(f\left(\frac{x+y+z}{2}\right)+f\left(\frac{x-y+z}{2}\right)-f(x)-f(z),t \right)\\
\geq
N'\left(\theta\left(\|x\|^{p}+\|y\|^{p}+\|z\|^{p}\right),t\right)
\end{eqnarray*}
for all $x,y,z\in X$ and $t>0$. Then there is a
unique additive mapping $A:X\rightarrow Y$ that satisfies
(\ref{main}) and the inequality
\begin{equation*}
N(f(x)-A(x),t)\geq N'\left(\frac{(2^r+2)\theta
\|x\|^{p}}{2},t\right)
\end{equation*}
\end{cor}
\begin{proof}
Let
$\varphi(x,y,z):=\theta(\|x\|^{p}+\|y\|^{p}+\|z\|^{p})$
and $|r|=\frac{1}{4}$. Applying Theorem \ref{th41}, we get the
desired results.
\end{proof}
\begin{thm}\label{th42}
Assume that a mapping $f:X\rightarrow Y$ satisfies
the inequality (\ref{18}) and $\varphi:X^2\rightarrow Z$ is a
mapping for which there is a constant $r\in \mathbb R$ satisfying
$0<|r|<2$ such that
\begin{equation}\label{19}
N'\left(\varphi(x,y,z),|r|t\right)\geq
N'\left(\varphi\left(\frac{x}{2},\frac{y}{2},\frac{z}
{2}\right),t\right)
\end{equation}
for all $x,y,z\in X$ and all $t>0$. Then there exists a
unique additive mapping $A:X\rightarrow Y$  satisfying
(\ref{main}) and the following inequality
\begin{equation}\label{20qq}
N(f(x)-A(x),t)\geq
N'\left(\frac{\varphi(x,2x,x)}{2-|r|},t\right)\:,
\end{equation}
for all $x\in X$ and all $t>0$.
\end{thm}
\begin{proof}
It follows from (\ref{uuu}) that
\begin{equation}\label{hhh}
N\left(\frac{f(2x)}{2}-f(x),\frac{t}{2}\right)\geq
N'(\varphi(x,2x,x),t)
\end{equation}
for all $x\in X$ and all $t>0$. Replacing $x$ by $2^{n}x$ in
(\ref{hhh}), we obtain
\begin{eqnarray}\nonumber\label{kkk}
N\left(\frac{f(2^{n+1}x)}{2^{n+1}}-\frac{f(2^{n}x)}{2^{n}},
\frac{t}{2^{n+1}}\right)&\geq&
N'(\varphi(2^{n}x,2^{n+1}x,2^{n}x),t)\\
 &\geq&
N'\left(\varphi(x,2x,x),\frac{t}{|r|^{n}}\right).
\end{eqnarray}
Thus
\begin{eqnarray}\label{dd}
&&N\left(\frac{f(2^{n+1}x)}{2^{n+1}}-\frac{f(2^{n}x)}{2^{n}},
\frac{|r|^{n}t}{2^{n+1}}\right)\geq
N'(\varphi(x,2x,x),t)
\end{eqnarray}
for all $x\in X$ and all $t>0$. Proceeding as in the proof of
Theorem \ref{th41}, we obtain that
$$N\left(f(x)-\frac{f(2^{n}x)}{2^{n}},
 \sum_{j=0}^{n-1}\frac{|r|^{j}t}{2^{j+1}}\right)\geq N'(\varphi(x,2x,x),t)$$
for all $x\in X$, all $t>0$ and any integer $n>0$. Therefore
\begin{eqnarray}\nonumber
N\left(f(x)-\frac{f(2^{n}x)}{2^{n}},t\right)&\geq&
N'\left(\varphi(x,2x,x),\frac{t}{
\sum_{j=0}^{n-1}\frac{|r|^{j}}{2^{j+1}}}\right)\\ &\geq&
N'\left(\varphi(x,2x,x),(2-|r|)t\right).
\end{eqnarray}
The rest of the proof is similar to the proof of Theorem
\ref{th41}.
\end{proof}
\begin{cor}
Let $X$ be a normed space and  $(\mathbb R,N')$ be a fuzzy
Banach space. Assume that there exists real number $\theta\geq 0$
and $0<p=p_1+p_2+p_3<2$ such that a mapping $f:X\rightarrow
Y$ satisfies the following inequality
\begin{eqnarray}\nonumber
N\left(f\left(\frac{x+y+z}{2}\right)+f\left(\frac{x-y+z}{2}\right)-f(x)-f(z),t \right)\\ \nonumber
\geq
N'\left(\theta\left(\|x\|^{p_1}\cdot\|y\|^{p_2}\cdot\|z\|^{p_3}\right),t\right)
\end{eqnarray}
for all $x,y,z\in X$ and $t>0$. Then there is a
unique additive mapping $A:X\rightarrow Y$  satisfying
(\ref{main}) and the inequality
\begin{equation*}
N(f(x)-A(x),t)\geq N'\left((2^r+2)\theta
\|x\|^{p},t\right)
\end{equation*}
\end{cor}
\begin{proof}
Let
$\varphi(x,y,z):=\theta\left(\|x\|^{p_1}\cdot\|y\|^{p_2}\cdot\|z\|^{p_3}\right)$
and $|r|=1$. Applying Theorem \ref{th42}, we get the
desired results.
\end{proof}

\subsection{\bf A fixed point method}

\begin{equation*}\end{equation*}
Throughout this subsection, using the fixed point alternative approach we
prove the Hyers-Ulam-Rassias stability of the functional equation
(\ref{main}) in fuzzy Banach spaces. In this subsection, we assume
that $X$ is a vector space and that $(Y,N)$ is a fuzzy Banach
space.
\begin{thm}\label{th4.5}
Let $\varphi: X^{3}\rightarrow [0,\infty)$ be a function such that
there exists an $L<1$ with
\begin{equation*}
\varphi\left(\frac{x}{2},\frac{y}{2},\frac{z}{2}\right)\leq
\frac{L \varphi(x,y,z)}{2}
\end{equation*}
for all $x,y,z\in X$. Let $f:X\rightarrow Y$ is a mapping  satisfying
\begin{eqnarray}\label{2}
&&N\left(f\left(\frac{x+y+z}{2}\right)+f\left(\frac{x-y+z}{2}\right)-f(x)-f(z),t \right)\\
\nonumber &&\geq \frac{t}{t+\varphi(x,y,z)}
\end{eqnarray}
for all $x,y,z\in X$ and all $t>0$ . Then, the limit
$$
A(x):=N\text{-}\lim_{n\rightarrow \infty} 2^n
f\left(\frac{x}{2^n} \right)\:,\ \ \text{exists for each $x\in X$}
$$
 and defines a unique additive mapping
$A:X\rightarrow Y$ such that
\begin{equation}\label{7}
N(f(x)-A(x),t)\geq
\frac{(2-2L)t}{(2-2L)t+ L
\varphi(x,2x,x)} .
\end{equation}
\end{thm}

\begin{proof}
Setting $y=2x$ and $z=x$ in (\ref{2}) and replacing $x$ by $\frac{x}{2}$, we obtain
\begin{equation}\label{oo}
N\left(2f\left(\frac{x}{2}\right)-f(x),t\right)\geq
\frac{t}{t+\varphi\left(\frac{x}{2},x,\frac{x}{2}\right)}
\end{equation}
for all $x\in X$ and $t>0$. Consider the set
$
S:=\{g:X\rightarrow Y\}
$
and the generalized metric $d$ in $S$ defined by
\begin{equation*}
d(f,g)=\inf \Big\{\mu\in \mathbb{R}^+: N(g(x)-h(x),\mu t)\geq
\frac{t}{t+\varphi(x,2x,x)}, \forall x\in X,\, t>0 \Big\},
\end{equation*}
where $\inf\,\emptyset=+\infty$. It is easy to show that $(S,d)$
is complete (see \cite[Lemma 2.1]{mr}). Now, we consider a linear
mapping $J:S\rightarrow S$ such that
\begin{equation*}
Jg(x):=2g\left(\frac{x}{2}\right)
\end{equation*}
for all $x\in X$. Let $g,h\in S$ be  such that $d(g,h)=\epsilon$.
Then
\begin{equation*}
N(g(x)-h(x),\epsilon t)\geq \frac{t}{t+\varphi(x,2x,x)}
\end{equation*}
for all $x\in X$ and $t>0$. Hence
\begin{eqnarray*}\nonumber
N(Jg(x)-Jh(x),L\epsilon t)&=&
N\left(2g\left(\frac{x}{2}\right)-2h\left(\frac{x}{2}\right)
,L\epsilon t\right)\\
&=&
N\left(g\left(\frac{x}{2}\right)-h\left(\frac{x}{2}\right),
\frac{L\epsilon t}{2}\right)\geq \frac{\frac{Lt}{2}}{\frac{Lt}{2}+
\varphi\left(\frac{x }{2},x,\frac{x}{2}\right)}\\
\nonumber &\geq& \frac{\frac{Lt}{2}}{\frac{Lt}{2}+
\frac{L \varphi(x,2x,x)}{2}}= \frac{t}{t+\varphi(x,2x,x)}
\end{eqnarray*}
for all $x\in X$ and  $t>0$. Thus $d(g,h)=\epsilon$ implies that
$d(Jg,Jh)\leq L\epsilon$. This means that
$$
d(Jg,Jh)\leq L d(g,h)\:,\ \ \text{for all $g,h\in S$.}
$$
 It follows from (\ref{oo}) that
\begin{eqnarray}\nonumber
N\left(f(x)-2f\left(\frac{x}{2}\right),t\right)
&\geq& \frac{t}{t+\varphi\left(\frac{x}{2},x,\frac{x}{2}\right)} \geq \frac{t}{t+\frac{L\varphi(x,2x,x)}{2}} = \frac{\frac{2t}{L}}{\frac{2t}{
L}+\varphi(x,2x,x)}.
\end{eqnarray}
Therefore
\begin{equation}
N\left(f(x)-2f\left(\frac{x}{2}\right),\frac{Lt}{2}\right)
\geq \frac{t}{t+\varphi(x,2x,x)}.
\end{equation}
Hence
$$d(f,Jf)\leq \frac{L}{2}.$$ 
By Theorem \ref{thras}, there exists a mapping $A:X\rightarrow Y$
satisfying the following:\\
(1) $A$ is a fixed point of $J$, that is,
\begin{equation}\label{11}
A\left(\frac{x}{2}\right)=\frac{A(x)}{2}\:,\ \ \text{for all $x\in X$. }
\end{equation}
The mapping $A$ is a unique fixed point of $J$
in the set
$$
\Omega=\{h\in S: d(g,h)<\infty\}.
$$
This implies that $A$ is a unique mapping satisfying (\ref{11})
such that there exists $\mu\in (0,\infty)$ with
\begin{equation*}
N(f(x)-A(x),\mu t)\geq \frac{t}{t+\varphi(x,2x,x)}
\end{equation*}
for all $x\in X$ and $t>0$.\\
(2) $d(J^n f, A)\rightarrow 0$ as $n\rightarrow \infty$. This
implies the equality
$$
N\text{-}\lim_{n\rightarrow \infty}2^n
f\left(\frac{x}{2^n} \right)= A(x)\:,\ \ \text{for all $x\in X$.}
$$
(3) $d(f,A)\leq \frac{d(f,Jf)}{1-L}$ with $f\in \Omega$, which
implies the inequality
$$
d(f,A)\leq \frac{L}{2-2L}.
$$
This implies that the inequality (\ref{7}) holds. Furthermore,
since
 \begin{eqnarray*}
&&N\left(A\left(\frac{x+y+z}{2}\right)+A\left(\frac{x-y+z}{2}\right)-A(x)-A(z),t \right)\\
&&\geq N-\lim_{n\rightarrow \infty}
\left(2^{n}f\left(\frac{x+y+z}{2^{n+1}}\right)+2^n f\left(\frac{x-y+z}{2^{n+1}}\right)-2^n f\left(\frac{x}{2^{n}}\right)-2^n f\left(\frac{z}{2^{n}}\right),t\right)\\ &&\geq \lim_{n\rightarrow
\infty}\frac{\frac{t}{2^{n}}}{\frac{t}{2^{n}}+\varphi\left(\frac{x}{2^n},
\frac{y}{2^n},\frac{z}{2^n}\right)}
\geq \lim_{n\rightarrow \infty}\frac{\frac{t}{2^{n}}}{\frac{t}{2^{n}}+
\frac{L^{n}\varphi(x,y,z)}{2^{n}}}\rightarrow 1
\end{eqnarray*}
 for all $x,y,z\in X$, $t>0$. It follows that
$$N\left(A\left(\frac{x+y+z}{2}\right)+A\left(\frac{x-y+z}{2}\right)-A(x)-A(z),t \right)=1$$
for all $x,y,z\in X$ and all $t>0$. Thus the mapping
$A: X\rightarrow Y$ is additive, as desired.
\end{proof}

\begin{cor}
Let $\theta\geq 0$ and let $p$ be a real number with $p>1$. Let
$X$ be a normed vector space with norm $\|.\|$. Let $f:
X\rightarrow Y$ be a mapping satisfying
\begin{eqnarray}\nonumber
N\left(f\left(\frac{x+y+z}{2}\right)+f\left(\frac{x-y+z}{2}\right)-f(x)-f(z),t \right)\\ \nonumber \geq
\frac{t}{t+\theta\left(\|x\|^p+\|y\|^p+\|z\|^p\right)}
\end{eqnarray}
for all $x,y,z\in X$ and all $t>0$. Then, the limit
\begin{equation}\nonumber
A(x):=N\text{-}\lim_{n\rightarrow\infty}2^{n}f\left(\frac{x}{2^{n}}\right)
\end{equation}
exists for each $x\in X$ and defines a unique additive mapping $A:
X\rightarrow Y$ such that
\begin{equation*}
N(f(x)-A(x),t)\geq \frac{(2^{p+1}-2)t}{(2^{p+1}-2)t+
(2^{r}+2)\theta \|x\|^p}
\end{equation*}
for all $x\in X$.
\end{cor}

\begin{proof}
The proof follows from Theorem \ref{th4.5} by taking $$
\varphi(x,y,z):=\theta(\|x\|^p+\|y\|^p
+\|z\|^p)\:,\ \  \text{for all $x,y,z\in X$.}$$ 
Then we can choose
$L=2^{-p}$ and we get the desired result.
\end{proof}
\begin{thm}\label{xc}
Let $\varphi: X^{3}\rightarrow [0,\infty)$ be a function such that
there exists an $L<1$ with
\begin{equation*}
\varphi(x,y,z)\leq
2L\varphi\left(\frac{x}{2},\frac{y}{2},\frac{z}{2}\right)
\end{equation*}
for all $x,y,z\in X$. Let $f:X\rightarrow Y$ be a
mapping satisfying (\ref{2}). Then $$
A(x):=N\text{-}\lim_{n\rightarrow \infty}\frac{f(2^n
x)}{2^n} $$ exists for each $x\in X$ and defines a unique
additive mapping $A:X\rightarrow Y$ such that
\begin{equation}\label{14}
N(f(x)-A(x),t)\geq \frac{(2-2L)t}{(2-2L)t+
\varphi(x,2x,x)}
\end{equation}
for all $x\in X$ and all $t>0$.
\end{thm}
\begin{proof}
Let $(S,d)$ be the generalized metric space defined as in the
proof of Theorem 4.1. Consider the linear mapping $J:S\rightarrow
S$ such that $$ Jg(x):=\frac{g(2x)}{2}\:,\ \ \text{for all $x\in X$.}$$
Let $g,h\in S$ be  such that $d(g,h)=\epsilon$. Then
\begin{equation*}
N(g(x)-h(x),\epsilon t)\geq \frac{t}{t+\varphi(x,2x,x)}
\end{equation*}
for all $x\in X$ and $t>0$ . Hence
\begin{eqnarray*}\nonumber
N(Jg(x)-Jh(x),L\epsilon t)&= & N\left(\frac{g(2x)}{2}
-\frac{h(2x)}{2},L\epsilon t\right)\\
&=& N\Big(g(2x)-h(2x),
2L\epsilon t\Big)\geq \frac{2L t}{2L t+\varphi(2x,
,4x,2x)}\\
\nonumber &\geq& \frac{2L t}{2L t+2L\varphi(x
,2x,x)}= \frac{t}{t+\varphi(x,2x,x)}
\end{eqnarray*}
for all $x\in X$ and  $t>0$. Thus 
$$d(g,h)=\epsilon\ \  \text{implies that
$d(Jg,Jh)\leq L\epsilon$.}$$ 
This means that
\begin{equation*}
d(Jg,Jh)\leq L d(g,h)\:,\ \ \text{for all $g,h\in S$. }
\end{equation*}
It follows from (\ref{oo}) that
\begin{eqnarray}\nonumber
N\left(\frac{f(2x)}{2}-f(x),\frac{t}{2}\right)
\geq\frac{t}{t+\varphi(x,2x,x)}.
\end{eqnarray}
Therefore $$d(f,Jf)\leq \frac{1}{2}.$$
 By Theorem \ref{thras}, there
exists a mapping $A:X\rightarrow Y$ satisfying the following:\\
(1) $A$ is a fixed point of $J$, that is,
\begin{equation}\label{16}
2A(x)=A(2x)\:,\ \ \text{for all $x\in X$. }
\end{equation}
The mapping $A$ is a unique fixed point of $J$
in the set
$$
\Omega=\{h\in S: d(g,h)<\infty\}.
$$
This implies that $A$ is a unique mapping satisfying (\ref{16})
such that there exists $\mu\in (0,\infty)$ satisfying
\begin{equation*}
N(f(x)-A(x),\mu t)\geq \frac{t}{t+\varphi(x,2x,x)}
\end{equation*}
for all $x\in X$ and $t>0$.\\
 (2) $d(J^n f, A)\rightarrow 0$ as
$n\rightarrow \infty$. This implies the equality $$
N\text{-}\lim_{n\rightarrow \infty}\frac{f(2^{n}x)}{2^{n}}\:,\ \  \text{for all $x\in X$.}$$
(3) $d(f,A)\leq \frac{d(f,Jf)}{1-L}$ with $f\in \Omega$, which
implies the inequality
$$
d(f,A)\leq \frac{1}{2-2L}.
$$
This implies that the inequality (\ref{14}) holds. The rest of the
proof is similar to that of the proof of Theorem \ref{th4.5}.
\end{proof}

\begin{cor}
Let $\theta\geq 0$ and let $p$ be a real number with
$0<p<\frac{1}{3}$. Let $X$ be a normed vector space with norm
$\|.\|$. Let $f: X\rightarrow Y$ be a mapping
satisfying
\begin{eqnarray}\nonumber
N\left(f\left(\frac{x+y+z}{2}\right)+f\left(\frac{x-y+z}{2}\right)-f(x)-f(z),t \right)\\ \nonumber \geq
\frac{t}{t+\theta\left(\|x\|^p.\|y\|^p.\|z\|^p\right)}
\end{eqnarray}
for all $x,y,z\in X$ and all $t>0$. Then
\begin{equation*}\nonumber
A(x):=N\text{-}\lim_{n\rightarrow\infty}\frac{f(2^{n}x)}{2^{n}}\:,\ \text{exists for each $x\in X$}
\end{equation*}
 and defines a unique additive mapping $A:
X\rightarrow Y$ such that
\begin{equation*}
N(f(x)-A(x),t)\geq
\frac{(2^{1+3p}-2)t}{(2^{1+3p}-2)t+ 2
^{3p}\theta \|x\|^{3p}}\:,\ \ \text{for all $x\in X$.}
\end{equation*}
\end{cor}

\begin{proof}
The proof follows from Theorem \ref{xc} by taking $$
\varphi(x,y,z):=
\theta\left(\|x\|^p.\|y\|^p.\|z\|^p\right)$$
for all $x,y,z\in X$. Then, we can choose
$L=2^{-3p}$ and we get the desired result.
\end{proof}



\bigskip
\bigskip


{\footnotesize \pn{\bf H. Azadi Kenary}\; \\ {Department of
Mathematics}, {College of Sciences, Yasouj University
, Yasouj 75914-353,} {Iran.}\\
{\tt Email: azadi@yu.ac.ir}\\

{\footnotesize \pn{\bf Themistocles M. Rassias}\; \\ { Department of
Mathematics
 National Technical University of Athens
 Zografou Campus,
 157 80, Athens
} { GREECE.}\\
{\tt Email: trassias@math.ntua.gr}\


\begin{thebibliography}{20}

\bibitem{mr4} M. R. Abdollahpour, R. Aghayaria and M. Th. Rassias, \textit{Hyers-Ulam stability of associated Laguerre differential equations in a subclass of analytic functions}, Journal of Mathematical Analysis and Applications, 437(2016), 605-612.

\bibitem{mr2} M. R. Abdollahpour and M. Th. Rassias, \textit{Hyers-Ulam stability of hypergeometric differential equations}, Aequationes Mathematicae, (93)(4)(2019), 691-698.

\bibitem{a} T. Aoki, \textit{On the stability of the linear transformation
in Banach spaces}, J. Math. Soc.  Japan {\bf 2} (1950), 64--66.


\bibitem{ab} L.M. Arriola and W.A. Beyer, \textit{Stability of the Cauchy
functional equation over $p$-adic fields}, Real Anal. Exchange
{\bf 31} (2005/06),  125--132.

\bibitem{ak1} H. Azadi Kenary, \textit{On the stability of a cubic
functional equation in random normed spaces,} J.  Math. Extension
{\bf  4} (2009), 1--11.

\bibitem{ak3} H. Azadi Kenary, \textit{Stability of a
Pexiderial functional equation in random normed spaces,} Rend.
Circ. Mat. Palermo, (2011) DOI:10.1007/s12215-011-0027-5.

\bibitem{B-S}
T. Bag and S.K. Samanta, \textit{Finite dimensional fuzzy normed
linear spaces}, Journal of Fuzzy Mathematics {\bf  11} (2003),
687--705.

\bibitem{B-S2}
T. Bag and S.K. Samanta, \textit{Fuzzy bounded linear operators},
Fuzzy Sets and Systems {\bf 151} (2005), 513--547.


     \bibitem {cr1} L. C\u{a}dariu and V. Radu, \newblock{\it Fixed points and the stability of Jensen's functional equation},
\newblock J. Inequal. Pure Appl. Math. {\bf 4}, no. 1, Art. ID 4 (2003).

\bibitem{C-M}
S.C. Cheng and J.N. Mordeson, \textit{Fuzzy linear operators and
fuzzy normed linear spaces}, Bulletin of Calcutta Mathematical
Society {\bf 86} (1994), 429--436.


\bibitem{ch} P.W. Cholewa, {\it Remarks on the stability of functional
equations,} Aequationes Math. {\bf 27} (1984), 76--86.

\bibitem{cs} J. Chung and P.K. Sahoo, \textit{On the general solution
of a quartic functional equation,} Bull. Korean Math. Soc. {\bf
40} (2003),  565--576.

\bibitem{c} S. Czerwik, \textit{Functional equations and inequalities in
several variables}, World Scientific, River Edge, NJ, 2002.

\bibitem {cr2} L. C\u{a}dariu and V. Radu, \newblock{\it On the stability of the Cauchy functional equation: a fixed point approach},
 \newblock Grazer Math. Ber. {\bf 346} (2004), 43--52.

  \bibitem {cr3} L. C\u{a}dariu and V. Radu, \newblock{\it Fixed point methods for the generalized stability of functional equations in a single variable},
 \newblock Fixed Point Theory and Applications {\bf 2008}, Art. ID 749392 (2008).

\bibitem{d} D. Deses, \textit{On the representation of non-Archimedean
objects}, Topology Appl. {\bf 153} (2005),  774--785.

 \bibitem {dm} J. Diaz and B. Margolis, \newblock{\it A fixed point theorem of the alternative for contractions on a generalized
 complete metric space},  \newblock Bull. Amer. Math. Soc. {\bf 74} (1968), 305--309.


\bibitem{ep} M. Eshaghi-Gordji, S. Abbaszadeh, and C. Park,
\textit{On the stability of a generalized quadratic and quartic
type functional equation in quasi-Banach spaces,} J. Inequal.
Appl. {\bf  2009} (2009), Article ID 153084, 26 pages.

\bibitem{eshaghi-bavand} M. Eshaghi Gordji and M. Bavand Savadkouhi, \emph{Stability of mixed type
cubic and quartic functional equations in random normed spaces,}
J. Ineq. Appl., Vol. 2009(2009), Article ID 527462, 9 pages.

\bibitem{eshaghi-bavand-park} M. Eshaghi Gordji and M. Bavand Savadkouhi and Choonkil Park, \emph{Quadratic-quartic functional equations in
RN-spaces,} J. Ineq. Appl., Vol. 2009(2009), Article ID 868423, 14
pages.

\bibitem{eshaghi-khodae} M. Eshaghi Gordji and H. Khodaei, \emph{Stability of functional equations,}
Lap Lambert Academic Publishing, 2010.

\bibitem{eshaghi-zolfaghari-rassias-Bavand} M. Eshaghi Gordji, S. Zolfaghari,
 J.M. Rassias and M.B. Savadkouhi, \emph{Solution and stability of a mixed type cubic and quartic  functional equation in
quasi-Banach spaces,} Abst. Appl. Anal., Vol. 2009(2009), Article
ID 417473, 14 pages.




\bibitem{f}  W. Fechner, \textit{Stability of a functional inequality
associated with the Jordan-von Neumann functional equation},
Aequationes Math. {\bf 71} (2006), 149--161.

\bibitem{g}  P. G\v{a}vruta, \textit{A generalization of the
Hyers-Ulam-Rassias stability of approximately additive mappings},
J. Math. Anal. Appl. {\bf 184} (1994),  431--436.

\bibitem{h}  K. Hensel, \textit{Ubereine news Begrundung der Theorie der
algebraischen Zahlen}, Jahresber. Deutsch. Math. Verein {\bf 6}
(1897), 83--88.

\bibitem{hy}  D.H. Hyers, \textit{On the stability of the linear functional
equation}, Proc. Nat. Acad. Sci. U.S.A. {\bf 27} (1941), 222--224.

\bibitem{hir}  D.H. Hyers, G. Isac, and Th.M. Rassias, \textit{Stability of
Functional Equations in Several Variables}, Birkh\"{a}user, Basel,
1998.

\bibitem{jkr} K. Jun, H. Kim and J.M. Rassias, \textit{Extended
Hyers-Ulam stability for Cauchy-Jensen mappings}, J. Difference
Equ. Appl. {\bf 13} (2007),  1139--1153.

\bibitem{mr8} S. -M. Jung, D. Popa and M. Th. Rassias, \textit{On the stability of the linear functional equation in a single variable on complete metric groups}, Journal of Global Optimization, 59(2014), 165-171.

\bibitem{mr7} S. -M. Jung and M. Th. Rassias, \textit{A linear functional equation of third order associated to the Fibonacci numbers}, Abstract and Applied Analysis, Volume 2014 (2014), Article ID 137468, 7 pages.

\bibitem{mr5} S. -M. Jung, C. Mortici and M. Th. Rassias,  \textit{On a functional equation of trigonometric type}, Applied Mathematics and Computation, 252(2015), 294-303.

\bibitem{Kat}
A.K. Katsaras, \textit{Fuzzy topological vector spaces}, Fuzzy
Sets and Systems {\bf  12} (1984), 143--154.


\bibitem{kb}  A. K. Katsaras and A. Beoyiannis, \textit{Tensor products of
non-Archimedean weighted spaces of continuous functions}, Georgian
Math. J. {\bf 6} (1999), 33--44.

\bibitem{kar-J.M}
I. Karmosil and J. Michalek, \textit{Fuzzy metric and statistical
metric spaces}, Kybernetica {\bf 11} (1975), 326--334.

\bibitem{k-s}
S.V. Krishna and K.K.M. Sarma, \textit{Separation of fuzzy normed
linear spaces}, Fuzzy Sets and Systems {\bf  63} (1994), 207--217.


\bibitem{kh} A. Khrennikov, \textit{Non-Archimedean Analysis: Quantum
Paradoxes, Dynamical Systems and Biological Models}, Mathematics
and its Applications {\bf 427}, Kluwer Academic Publishers,
Dordrecht, 1997.

\bibitem{ko} Z. Kominek, \textit{On a local stability of the Jensen
functional equation}, Demonstratio Math. {\bf 22} (1989),
499--507.

\bibitem{mr9} Y. -H. Lee, S. -M. Jung and M. Th. Rassias, \textit{On an n-dimensional mixed type additive and quadratic functional equation}, Applied Mathematics and Computation, 228(2014), 13-16.

\bibitem{mr3} Y. -H. Lee, S. -M. Jung and M. Th. Rassias, \textit{Uniqueness theorems on functional inequalities concerning cubic-quadratic-additive equation}, Journal of Mathematical Inequalities, 12(1)(2018), 43-61.



\bibitem{hir}  D.H. Hyers, G. Isac, and Th.M. Rassias, \textit{Stability of
Functional Equations in Several Variables}, Birkh\"{a}user, Basel,
1998.


\bibitem{mr}
D. Mihet and V. Radu, \textit{On the stability of the additive
Cauchy functional equation in random normed spaces}, J. Math.
Anal. Appl. {\bf  343} (2008), 567--572.

\bibitem{ms} M. Mohammadi, Y.J. Cho, C. Park, P. Vetro and R. Saadati,
\textit{Random stability of an additive-quadratic-quartic
functional equation}, J. Inequal. Appl. {\bf 2010} (2010), Article
ID 754210, 18 pages.

\bibitem{mr6} C. Mortici, S. -M. Jung and M. Th. Rassias, \textit{On the stability of a functional equation associated with the Fibonacci numbers}, Abstract and Applied Analysis, Volume 2014 (2014), Article ID 546046, 6 pages.

\bibitem{np} A. Najati and C. Park, \textit{The Pexiderized Apollonius-Jensen type additive mapping and isomorphisms between $C^*$-algebras},
J. Difference Equ. Appl. {\bf 14} (2008),  459--479.

\bibitem{n} P.J. Nyikos, \textit{On some non-Archimedean spaces of
Alexandrof and Urysohn}, Topology Appl. {\bf 91} (1999), 1--23.


\bibitem{c-3}
C. Park, \textit{Fuzzy stability of a functional equation
associated with inner product spaces}, Fuzzy Sets and Systems {\bf
160} (2009), 1632--1642.

\bibitem{p} C. Park, \textit{Generalized Hyers-Ulam-Rassias stability of $n$-sesquilinear-quadratic mappings on Banach modules over $C^*$-algebras},
J. Comput. Appl. Math. {\bf 180} (2005),  279--291.

 \bibitem {paf} C. Park, \newblock{\it Fixed points and Hyers-Ulam-Rassias stability of Cauchy-Jensen
 functional equations in Banach algebras},
 \newblock Fixed Point Theory and Applications {\bf 2007}, Art. ID 50175 (2007).
 

\bibitem {pa8} C. Park, \newblock{\it Generalized Hyers-Ulam-Rassias
stability of quadratic functional equations: a fixed point
approach},
 \newblock Fixed Point Theory and Applications {\bf 2008}, Art. ID 493751 (2008).
 
 \bibitem{mr1} C. Park and M. Th. Rassias, \textit{Additive functional equations and partial multipliers in  C*-algebras}, Revista de la Real Academia de Ciencias Exactas, Serie A. Matemáticas, 113(3)(2019), 2261-2275.

\bibitem{pv} J.C. Parnami and H.L. Vasudeva, \textit{On Jensen's
functional equation}, Aequationes Math. {\bf 43} (1992), 211--218.

  \bibitem {r}  V. Radu, \newblock{\it The fixed point alternative and the stability of functional equations},  \newblock
Fixed Point Theory  {\bf 4} (2003), 91--96.


\bibitem{ra} Th.M. Rassias, \textit{On the stability of the linear mapping
in Banach spaces}, Proc. Amer. Math. Soc. {\bf 72} (1978),
297--300.
\bibitem{ras} Th. M. Rassias, \emph{On the stability of the linear mapping
in Banach spaces}, Proc. Amer. Math. Soc. \textbf{72} (1978),  297-300.

\bibitem{ras2} Th. M. Rassias, \emph{Functional Equations,
Inequalities and Applications}, Klower Academic Publishers Co.,
Dordrecht, Boston, London, 2003.

\bibitem{ras3} Th. M. Rassias, \emph{Problem 16;2, Report of the
27th International Symp. on Functional Equations}, Aequations
Math., \textbf{39} (1990), 292-293.

\bibitem{ras4} Th. M. Rassias, \emph{On the stability of the quadratic functional equation and its applications}, Studia
Univ. Babes-Bolyai. XLIII (1998) 89-124.

\bibitem{ras5} Th. M. Rassias, \emph{The problem of S.M. Ulam for approximately multiplicative mappings}, J. Math.
Anal. Appl. \textbf{246} (2000) 352-378.

\bibitem{ras6} Th. M. Rassias, \emph{On the stability of functional equations in Banach spaces}, J. Math. Anal. Appl.
\textbf{251} (2000) 264-284.

\bibitem{ras-semrl1} Th. M. Rassias and P. Semrl, \emph{On the behaviour of mappings which do not satisfy Hyers-Ulam
stability}, Proc. Amer. Math. Soc. \textbf{114} (1992) 989993.

\bibitem{ras-semrl2} Th. M. Rassias and P. Semrl, \emph{On the Hyers-Ulam stability of linear mappings}, J. Math. Anal.
Appl. \textbf{173} (1993) 325-338.


\bibitem{r} J. R\"{a}tz, \textit{On inequalities associated with the
Jordan-von Neumann functional equation}, Aequationes Math. {\bf
66} (2003), 191--200.

\bibitem{sp} R. Saadati and C. Park, \textit{Non-Archimedean $\mathcal{L}$-fuzzy normed spaces and stability of  functional equations}
(in press).

\bibitem{saadati} R. Saadati, M. Vaezpour and Y. J. Cho, \emph{A note to paper
``On the stability of cubic mappings and quartic mappings in
random normed spaces"}, J. Ineq. Appl., Volume 2009, Article ID
214530, doi: 10.1155/2009/214530.

\bibitem{reza} R. Saadati, M. M. Zohdi, and S. M. Vaezpour, \textit{Nonlinear L-Random Stability of an ACQ Functional
Equation}, J. Ineq. and Appl., Volume 2011, Article ID 194394, 23
pages, doi:10.1155/2011/194394.

\bibitem{s-s} B. Schewizer and A. Sklar, \textit{Probabilistic Metric Spaces},
North-Holland Series in Probability and Applied Mathematics,
North-Holland, New York, USA, 1983.

\bibitem{s} F. Skof, \textit{Local properties and approximation of
operators}, Rend. Sem. Mat. Fis. Milano {\bf 53} (1983), 113--129.

\bibitem{u} S. M. Ulam, \textit{Problems in Modern Mathematics}, Science
Editions, John Wiley and Sons, 1964.

\end{thebibliography}
\end{document}